\documentclass[a4paper,11pt]{article}
\usepackage{bbding,textcomp,amsfonts,amsthm,amsmath,mathrsfs}
\usepackage{amssymb}
\usepackage{multirow, url}
\usepackage[utf8x]{inputenc}
\usepackage{comment}
\usepackage[affil-it]{authblk}
\usepackage{titling}
\usepackage[T1]{fontenc}
\usepackage{ae,aecompl}

\newtheorem{question}{Question}
\newtheorem{problem}{Problem}
\newtheorem{definition}{Definition}
\newtheorem{theorem}{Theorem}
\newtheorem{conjecture}[theorem]{Conjecture}
\newtheorem{corollary}[theorem]{Corollary}
\newtheorem{lemma}[theorem]{Lemma}
\newtheorem{proposition}[theorem]{Proposition}
\newtheorem{claim}{Claim}
\newtheorem{target}{Target}

\newcommand{\eps}{\varepsilon}

\newcommand{\fin}{\textrm{\upshape{fin}}}
\newcommand{\ex}{\textrm{\upshape{ex}}}

\newcommand{\By}[2]{\overset{\mbox{\tiny{#1}}}{#2}}
\newcommand{\ByRef}[2]{   \By{\eqref{#1}}{#2} }

\newcommand{\eqByRef}[1]{ \ByRef{#1}{=} }

\newcommand{\gByRef}[1]{  \ByRef{#1}{>} }
\newcommand{\leByRef}[1]{ \ByRef{#1}{\le} }
\newcommand{\geByRef}[1]{ \ByRef{#1}{\ge} }

\PrerenderUnicode{\unichar{355}}

\title{On the algebraic and topological structure of the set of Tur\'an densities}

\author{Codru\unichar{355} Grosu\thanks{This research was supported by the Deutsche Forschungsgemeinschaft within the research training group `Methods for Discrete Structures' (GRK 1408).}}
\affil{\small{grosu.codrut@gmail.com, Freie Universit\"at Berlin, Germany}}
\date{}

\begin{document}

\maketitle

\begin{abstract}The present paper is concerned with the various algebraic structures supported by the set of Tur\'an densities.

We prove that the set of Tur\'an densities of finite families of $r$-graphs is a non-trivial commutative semigroup, and as a consequence we construct explicit irrational densities for any $r \geq 3$. The proof relies on a technique recently developed by Pikhurko.

We also show that the set of all Tur\'an densities forms a graded ring, and from this we obtain a short proof of a theorem of Peng on jumps of hypergraphs.

Finally, we prove that the set of Tur\'an densities of families of $r$-graphs has positive Lebesgue measure if and only if it contains an open interval. This is a simple consequence of Steinhaus's theorem.
\end{abstract}

\tableofcontents

\section{\normalsize Introduction}

Let $r \geq 1$ and $\mathcal{F}$ be a (possibly infinite) family of $r$-graphs. 
For any $n \geq 1$, the Tur\'an function $\ex(n, \mathcal{F})$ is defined as the maximum possible number of edges of an $\mathcal{F}$-free $r$-graph on $n$ vertices (if no such $r$-graph exists, $\ex(n, \mathcal{F})=0$ by definition). The study of the Tur\'an function goes back to the foundational paper of Tur\'an \cite{Tur} from 1941, which essentially created the field of extremal graph theory. Most of the work concerning the Tur\'an function was devoted to the case of graphs ($r=2$), and not much is known about larger values of $r$ (see \cite{keevash11} for a survey). It is nevertheless the latter case that concerns us in this paper.

As was observed by Katona, Nemetz and Simonovits \cite{katona64}, one can define the Tur\'an density of $\mathcal{F}$ as
\begin{equation*}
\label{eq:turan}
\pi(\mathcal{F}) = \lim_{n \rightarrow \infty}\frac{\ex(n, \mathcal{F})}{\binom{n}{r}},
\end{equation*}
and this limit always exists. Let $\Pi_{\infty}^{(r)}$ consist of all possible Tur\'an densities of $r$-graph families and $\Pi_{\fin}^{(r)}$ be the set $\{\pi(\mathcal{F}) : \mathcal{F} \textrm{ is a finite family of $r$-graphs}\,\}$. Clearly $\Pi_{\fin}^{(r)} \subseteq \Pi_{\infty}^{(r)}$.

The Erd\H os-Stone-Simonovits theorem (\cite{ErdosStone1946}, \cite{Erdos66}) completely determines the set $\Pi_{\infty}^{(2)}$. In fact,
\begin{equation}
\label{eq:case2}
\Pi_{\fin}^{(2)} = \Pi_{\infty}^{(2)} = \{1\} \cup \{1 - \frac{1}{k} : k \geq 1\}.
\end{equation}

For $r \geq 3$, little is known about $\Pi_{\infty}^{(r)}$ or $\Pi_{\fin}^{(r)}$. Erd\H os \cite{ErdosC83} offered \$1000 for the complete determination of $\Pi_{\infty}^{(r)}$ for all $r$. Nevertheless, many basic questions are still open, perhaps the most famous one being Tur\'an's conjecture from 1941 that $\pi(K_4^{3}) = 5/9$ (here $K_t^{r}$ denotes the complete $r$-graph on $t$ vertices). Even more, to date no value $\pi(K_t^{r})$ for $3 \leq r < t$ has been determined.

One can roughly divide the present knowledge about Tur\'an densities into topological and algebraic facts. In order to obtain a better picture we shall try to summarize in the sequel some of the most important theorems.

One of the oldest results about $\Pi_{\infty}^{(r)}$ is due to Erd\H os \cite{Erdos64}, who proved that $\Pi_{\infty}^{(r)} \cap (0, r!/r^r) = \emptyset$. Erd\H os \cite{ErdosC83} went on to conjecture that for every $\alpha \in (0, 1)$ there exists some $\eps > 0$ with $\Pi_{\infty}^{(r)} \cap (\alpha, \alpha+\eps) = \emptyset$ (such an $\alpha$ is called \textit{a jump for $r$-graphs}). Clearly this is the case for $r=2$, albeit Frankl and R\"odl \cite{FranklRodl84} famously disproved the conjecture by showing that $1 - 1/\ell^{r-1}$ is a non-jump for $r$-graphs, for every $\ell > 2r, r \geq 3$. Erd\H os (\cite{Erdos71}\footnote{Incidentally this was the first paper in the first number of the journal \textit{Discrete Mathematics}.}, \cite{ErdosC83}) further conjectured that $r!/r^r$ is always a jump for $r$-graphs, and offered \$500 for a solution. This conjecture (called the \textit{jumping constant conjecture}) is still open. Many examples of non-jumps were constructed using the method of Frankl and R\"odl (\cite{peng07a}, \cite{peng07b}, \cite{peng08a}, \cite{peng08b}), the smallest of which is $\frac{5}{2}\frac{r!}{r^r}$ \cite{Frankl07}.

It was shown by Brown and Simonovits \cite{Brown84} that $\Pi_{\infty}^{(r)} \subseteq \overline{\Pi}_{\fin}^{(r)}$. Recently Pikhurko proved that equality holds here.

\begin{theorem}[Pikhurko, \cite{pikhurko12}]
The set $\Pi_{\infty}^{(r)}$ is closed in $[0,1]$.
\end{theorem}

Furthermore, Pikhurko proved the following.

\begin{theorem}[Pikhurko, \cite{pikhurko12}]
For every $r \geq 3$ the set $\Pi_{\infty}^{(r)}$ has cardinality of the continuum.
\end{theorem}

In particular, as $\Pi_{\fin}^{(r)}$ is countable, this means $\Pi_{\infty}^{(r)} \neq \Pi_{\fin}^{(r)}$ for $r \geq 3$.

It is an open question if $\Pi_{\infty}^{(r)}$ contains an interval of positive length for $r \geq 3$. Proving that a certain number \textit{does not belong} to $\Pi_{\infty}^{(r)}$ seems to be very hard. So far Baber and Talbot \cite{Baber11} proved that $[0.2299, 0.2316) \cap \Pi_{\infty}^{(3)} = \emptyset$, that $\pi(K_4^{3-})$ is a jump for $3$-graphs, and by upper-bounding $\pi(K_4^{3-})$, they proved that $[0.2871, 8/27) \cap \Pi_{\infty}^{(3)} = \emptyset$ (here $K_4^{3-}$ denotes the complete $3$-graph on $4$ vertices minus an edge). The proof uses flag algebras, introduced and developed by Razborov \cite{Razborov1}. Flag algebras have been successfully used for computing Tur\'an densities in certain special cases (\cite{Razborov2}, \cite{BaberT11}, \cite{ravry11}), and also for solving several open questions in graph theory (\cite{hladky12}, \cite{grzesik12}, \cite{hladky13}, \cite{kral12}, \cite{baber13}, \cite{roman13}).

This practically represents all that is known about the topological structure of $\Pi_{\infty}^{(r)}$.

On the algebraic side, it was proved by Baber and Talbot \cite{BaberT11} that $\Pi_{\fin}^{(3)}$ contains irrational numbers, disproving a conjecture of Chung and Graham \cite{ChungGraham}. Pikhurko independently proved the following more general result.

\begin{theorem}[Pikhurko, \cite{pikhurko12}]
\label{thm:pikirrat}
For every $r \geq 3$ the set $\Pi_{\fin}^{(r)}$ contains an irrational number.
\end{theorem}

The following question is due to Jacob Fox.

\begin{question}[Jacob Fox]
\label{ques:fox}
Does $\Pi_{\fin}^{(r)}$ contain a transcendental number?
\end{question}

Again this represents the current knowledge about the algebraic structure of $\Pi_{\fin}^{(r)}$.

Let us suppose for a moment that the answer to Question \ref{ques:fox} is negative. Then $\pi(\mathcal{F})$ is algebraic for any finite family $\mathcal{F}$. Any proof of this fact would likely supply a computable upper bound $K_{\mathcal{F}}$ on the degree and $\|\cdot\|_\infty$-norm of an integer polynomial having $\pi(\mathcal{F})$ as a root. Thus one could enumerate all integers polynomials bounded by $K_{\mathcal{F}}$, compute their roots, and obtain a finite set $S_{\mathcal{F}}$ containing $\pi(\mathcal{F})$. On the other hand, it is known that for any $\eps > 0$ and any finite family $\mathcal{F}$, one can approximate $\pi(\mathcal{F})$ up to an $\eps$ error by using a simple, albeit very inefficient algorithm (see \cite{keevash11}, page 3, for more details). Choosing $\eps$ small enough so that the distance between any two elements of $S_{\mathcal{F}}$ is at least $2\eps$, and applying the above algorithm, one could \textit{in principle} determine $\pi(\mathcal{F})$, thus solving the question of computing Tur\'an densities. Of course in practice this method would be too inefficient to be applied, but nevertheless it would provide a promising theoretical foundation for more sophisticated methods.

As a consequence we shall try to answer Question \ref{ques:fox} in the negative. Unfortunately proving that a certain number is algebraic is a difficult task. It would greatly help us in our endeavour if the set $\Pi_{\fin}^{(r)}$ would have a ring-like structure.

This suggests the following approach. Let $f \in \mathbb{Q}[x]$ be an arbitrary polynomial with rational coefficients and $\alpha$ a Tur\'an density. Under what conditions is $f(\alpha)$ a Tur\'an density? One of the goals of this paper is to give a partial answer to this question.

As nothing prevents us from taking the product $\pi(K_4^3)\pi(K_5^4)$, say, and this is still a real number\footnote{This number is believed to be $\frac{55}{144}$.} in $[0,1]$, it makes sense to first try to understand the effect of real multiplication on Tur\'an densities. It turns out that multiplying two Tur\'an densities comes close to giving another Tur\'an density, although one has to change the uniformity degree.

Formally, define \textit{the set of Tur\'an densities} as
\begin{equation*}
\Pi_{\infty} := \{(\alpha, r) : \alpha \in \Pi_{\infty}^{(r)}, r \geq 0\},
\end{equation*}
and \textit{the set of finite Tur\'an densities} as
\begin{equation*}
\Pi_{\fin} := \{(\alpha, r) : \alpha \in \Pi_{\fin}^{(r)}, r \geq 0\}.
\end{equation*}
For technical reasons we set here $\Pi_{\infty}^{(0)} = \Pi_{\fin}^{(0)} = \{1\}$.

We now define a binary operation $*$ on the set $\mathbb{R} \times \mathbb{N}$, which obviously contains $\Pi_{\infty}$:
\begin{align*}
  * \colon (\mathbb{R} \times \mathbb{N}) \times (\mathbb{R} \times \mathbb{N}) &\to \mathbb{R} \times \mathbb{N}\\
  (\alpha, r) \times (\beta, s) \phantom{xx} &\mapsto (\alpha\beta\binom{r+s}{r}\frac{r^rs^s}{(r+s)^{r+s}}, r+s).
\end{align*}

Our first result reads as follows.

\begin{theorem}
\label{thm:global}
$(\Pi_{\infty}, *)$ is a commutative cancellative monoid\footnote{Cancellative means $a+b = a+c$ implies $b=c$. This property ensures the existence of an embedding of the monoid in its Grothendieck group. One can always take the Grothendieck group of a monoid, but if it is not cancellative, the group may contain just one element.}.
\end{theorem}

In particular, $\mathbb{Z}\Pi_{\infty} \simeq \bigoplus_{r \geq 0} \mathbb{Z}\Pi_{\infty}^{(r)}$ is a graded ring under $*$ (here $\mathbb{Z}\Pi_{\infty}^{(r)}$ is the free abelian group generated by $\Pi_{\infty}^{(r)}$). Thus it makes sense to try to find a set of nontrivial relations which to quotient out from $\mathbb{Z}\Pi_{\infty}^{(r)}$. We will do this shortly, but first let us note the following corollary.

\begin{corollary}
\label{cor:jump}
Let $r \geq 3$ and $c > 0$. Suppose $c\frac{r!}{r^r}$ is a non-jump for $r$-graphs. Then $c\frac{q!}{q^q}$ is a non-jump for $q$-graphs, for any $q \geq r$.
\end{corollary}
\begin{proof}
Let $x_n$ be a sequence of elements in $\Pi_{\infty}^{(r)}$ converging to $c\frac{r!}{r^r}$ from above. Then $(1, 1) * (x_n, r) \in \Pi_{\infty}^{(r+1)}$ by Theorem \ref{thm:global} (here we abuse the notation slightly, as the result of $*$ is a pair, and not a number). Hence
\begin{equation*}
\lim_{n \rightarrow \infty} (1,1) * (x_n, r) = \lim_{n \rightarrow \infty} x_n\frac{r^r}{(r+1)^r} = c\frac{(r+1)!}{(r+1)^{r+1}}.
\end{equation*}
Thus $c\frac{(r+1)!}{(r+1)^{(r+1)}}$ is not a jump for $(r+1)$-graphs, and the claim follows by induction.
\end{proof}

Corollary \ref{cor:jump} was originally proved by Peng in \cite{peng09}.

A \textit{semiring} is a set $R$ with two binary operations $\otimes$ and $\oplus$ such that $(R, \oplus)$ is a commutative semigroup, $(R, \otimes)$ is a semigroup, and $\otimes$ is distributive over $\oplus$ (from left and from right). We do not require the existence of units. Furthermore $(R, \oplus, \otimes)$ is called \textit{commutative} if $(R, \otimes)$ is.

Now note that $\Pi_{\infty}^{(2)} = \{1-\frac{1}{k} : k \geq 1\} \cup \{1\}$ has the structure of a commutative semiring. This is obtained by transferring the semiring structure from $\mathbb{N}$ to $\Pi_{\infty}^{(2)}$ via the bijection $1 - \frac{1}{\ell} \rightarrow \ell$. Formally we define for all $a, b \geq 1$ and  $\alpha \in \Pi_{\infty}^{(2)}$,
\begin{align*}
(1 - \frac{1}{a}) \oplus_2 (1 - \frac{1}{b}) &= 1 - \frac{1}{a+b},\\
\alpha \oplus_2 1 = 1 \oplus_2 \alpha &= 1,
\end{align*}
and
\begin{align*}
(1 - \frac{1}{a}) \otimes_2 (1 - \frac{1}{b}) &= 1 - \frac{1}{ab}, \\
\alpha \otimes_2 1 = 1 \otimes_2 \alpha &= 1.
\end{align*}

Algebraically this means $\oplus_2$ maps $(\alpha, \beta)$ to the real number $1 - \frac{1-\alpha - \beta +\alpha\beta}{2 - \alpha -\beta}$, while $\otimes_2$ maps $(\alpha, \beta)$ to the real number $\alpha + \beta - \alpha\beta$, for any $\alpha, \beta \in \Pi_{\infty}^{(2)} \setminus \{1\}$. One may ask to what extent does this generalize to arbitrary $r$.

For any $r \geq	2$, let us define $\oplus_r : [0,1] \times [0,1] \rightarrow [0,1]$ by
\begin{equation}
\label{eq:addition}
\alpha \oplus_r \beta = 1 - \frac{1-\alpha-\beta+\alpha\beta}{(\sqrt[r-1]{1 - \alpha} + \sqrt[r-1]{1-\beta})^{r-1}},
\end{equation}
and $1 \oplus_r 1 = 1$. One sees that this is well-defined and for $r=2$ it agrees with the previous definition.

Our next result reads as follows.
\begin{theorem}
\label{thm:local}
$(\Pi_{\infty}^{(r)}, \oplus_r)$ is a commutative topological semigroup, and $\Pi_{\fin}^{(r)}$ is closed under $\oplus_r$.
\end{theorem}
Thus, for example, $\pi(K_4^{3})\oplus_3\pi(K_4^{3}) = \frac{3+\pi(K_4^{3})}{4}$ is the Tur\'an density of a finite family of $3$-graphs. Theorem \ref{thm:local} should be regarded as the main result of this paper\footnote{We remark that $\Pi_{\fin}^{(r)}$ has also a trivial semigroup structure given by the $\max$ operation.}.

As a consequence of Theorem~\ref{thm:local} one can construct explicit irrational densities for any $r \geq 3$.
\begin{corollary}
\label{cor:irrat1}
For every $r \geq 3$ the set $\Pi_{\fin}^{(r)}$ contains the irrational number
\begin{equation*}
\frac{r!}{r^r} \oplus_r 0 = 1 - \frac{r^{r-1} - (r-1)!}{\left(r+\sqrt[r-1]{r^{r-1}-(r-1)!}\right)^{r-1}}.
\end{equation*}
\end{corollary}
This in particular provides a new proof of Theorem~\ref{thm:pikirrat}. For even values of $r$ simpler examples can be given.
\begin{corollary}
\label{cor:irrat2}
For every even $r \geq 4$ the set $\Pi_{\fin}^{(r)}$ contains the irrational number $1 - \frac{1}{(1+\sqrt[r-1]{2})^{r-1}}$.
\end{corollary}
We defer the proofs of Corollaries~\ref{cor:irrat1} and~\ref{cor:irrat2} to Section $6$.

Combining Theorem~\ref{thm:global} with Theorem~\ref{thm:local} we obtain the following result.
\begin{corollary}
\label{cor:lebesgue}
For any $r \geq 2$ the following statements are equivalent:
\renewcommand{\theenumi}{(\roman{enumi})}
\begin{enumerate}
\item $\Pi_{\infty}^{(r)}$ has positive Lebesgue measure.
\item $\Pi_{\infty}^{(r)}$ contains an open interval.
\item For any $r' \geq r$, $\Pi_{\infty}^{(r')}$ has positive Lebesgue measure.
\item For any $r' \geq r$, $\Pi_{\infty}^{(r')}$ contains an open interval.
\end{enumerate}
\end{corollary}
\begin{proof}
We show that (i) implies (ii).

Set $A := \Pi_{\infty}^{(r)} \setminus \{1\}$ and note that $A$ is a Borel set (as $\Pi_{\infty}^{(r)}$ is closed) and still has positive Lebesgue measure. Furthermore $A$ is a semigroup under $\oplus_r$.

Define $\mathfrak{h} : [0,1) \rightarrow [1, +\infty)$ by $\mathfrak{h}(x) = \left(\frac{1}{1-x}\right)^{1/(r-1)}$. Then $\mathfrak{h}$ is a homeomorphism and also a semigroup isomorphism between $([0,1), \oplus_r)$ and $([1, +\infty), +)$. As $A$ is Borel, $\mathfrak{h}(A)$ is Lebesgue measurable.

The inverse of $\mathfrak{h}$ is $\mathfrak{h}^{-1}(x) = 1 - \frac{1}{x^{r-1}}$, which has first-order derivative $(\mathfrak{h}^{-1})'(x) = (r-1)\frac{1}{x^r}$. This is bounded on $[1, +\infty)$ and therefore $\mathfrak{h}^{-1}$ is Lipschitz. Consequently $\mathfrak{h}(A)$ has positive Lebesgue measure.

Hence by Steinhaus's theorem\footnote{Usually Steinhaus's theorem refers to the following statement: if $A \subset \mathbb{R}$ is a set of positive measure then $A - A$ contains an open interval. We use here a more general version: if $A, B \subset \mathbb{R}$ are sets of positive measure then $A+B$ contains an open interval. This is equivalent to \textit{Th\'eor\`eme VII} from Steinhaus's paper \cite{steinhaus1920}.}, $\mathfrak{h}(A) + \mathfrak{h}(A)$ contains an open interval. As $\mathfrak{h}(A)$ is a semigroup, $\mathfrak{h}(A)$ contains an open interval. Consequently $A$ contains an open interval, proving (ii).

We show that (ii) implies (iv).

Suppose $\Pi_{\infty}^{(r)}$ contains an open interval. Multiplying with $(1,1)$ as in the proof of Corollary \ref{cor:jump}, we obtain an open interval in $\Pi_{\infty}^{(r+1)}$. Then (iv) follows by induction.

All the other implications are trivial or follow from these two.
\end{proof}

Unfortunately $*$ is not distributive over $\oplus_r$, and this prevents us from giving $\Pi_{\infty}$ a more meaningful ring structure.

More surprisingly, we were not able to find a proper generalization of $\otimes_2$, as the natural approach fails in a non-trivial manner. Nevertheless, in the process we have arrived at the following conjecture.
\begin{conjecture}
\label{conj:closure}
$\overline{\cup_{r \geq 2} \Pi_{\fin}^{(r)}} = \overline{\cup_{r \geq 2} \Pi_{\infty}^{(r)}} = [0,1]$.
\end{conjecture}
A detailed discussion of this is given in Section 7. Along the way we shall encounter a generalization of van der Waerden's conjecture stated in terms of hypergraphs (this generalization was conjectured to hold by Dittert).

%Having a semiring structure on $\Pi_{\infty}^{(r)}$ would have most likely supplied us with the necessary tools to attack Question \ref{ques:fox}. Similarly if $\Pi_{\infty}$ would have turned out to have a richer algebraic structure. Unfortunately we were not able to verify either of these two hypotheses.
The rest of the paper is organised as follows.

In Section 2 we introduce the notation used in this paper. Section 3 presents an important technique for proving results about $\Pi_{\infty}^{(r)}$, which is then used in Section 4 to prove Theorem~\ref{thm:global}. Section 5 is devoted to the proof of Theorem~\ref{thm:local}, the main result of this paper. Section 6 contains the proof of Corollaries~\ref{cor:irrat1} and~\ref{cor:irrat2}. Section 7 discusses the problem of whether $\Pi_{\infty}^{(r)}$ is a semiring, leading in a natural way to Conjecture~\ref{conj:closure}. Finally, Section 8 lists several open problems.

\medskip
\textbf{Remark.} After completion of this work Oleg Pikhurko \cite{pikhurko15} found a short proof of Conjecture~\ref{conj:closure}. 

\section{\normalsize Notation}

We introduce some notation needed in the sequel.

An \textit{$r$-multiset} $D$ is an unordered collection of $r$ elements $x_1, \ldots , x_r$ with repetitions
allowed. The multiplicity $D(x)$ of $x$ in $D$ is the number of times that $x$ appears in $D$.

A pair $G = (V, E)$ with $E \subseteq V^{(r)}$ is called an \textit{$r$-multigraph}. $V$ is the set of \textit{vertices} and $E$ the set of \textit{edges}. Note that every edge is an $r$-multiset. Furthermore, note that our definition is different from the usual definition of a multigraph where the same edge may appear multiple times in $E$. If all edges in $E$ are proper sets, then $G$ is called a (simple) $r$-graph. We let $v(G) := |V(G)|$ be the number of vertices and $e(G) := |E(G)|$ be the number of edges. The \textit{density} of an $r$-graph $G$ is
\begin{equation*}
d(G) = \frac{e(G)}{\binom{n}{r}}.
\end{equation*}
We do not define the density of an $r$-multigraph.

We allow graphs without edges, and we also consider $\emptyset$ to be an $r$-graph without vertices. We call $\emptyset$ \textit{the empty graph}.

If $G$ and $H$ are $r$-graphs, we say $H$ is a \textit{subgraph} of $G$ if $V(H) \subseteq V(G)$ and $E(H) \subseteq E(G)$. $H$ is \textit{induced} if it has the edge set $\{X \in E(G) : X \subseteq V(H)\}$.

If $G$ is an $r$-graph and $U \subseteq V(G)$ we let $G[U]$ be the induced subgraph on vertex set $U$. If no confusion can arise, we may identify $G[U]$ with $U$. We let $G \setminus U$ denote the induced subgraph on $V(G) \setminus U$. Furthermore if $x \in V(G)$ we let $d_U(x) = d_{G[U]}(x)$ denote the degree of $x$ with respect to $U$, i.e. the number of edges of $G$ containing $x$ and intersecting $U \setminus \{x\}$ in $r-1$ vertices.

If $G$ and $H$ are $r$-graphs on disjoint vertex sets we let $G \dot\cup H$ denote the $r$-graph on vertex set $V(G) \dot\cup V(H)$ and edge set $E(G) \dot\cup E(H)$. We call $G\dot\cup H$ the \textit{disjoint union} of $G$ and $H$. As $G$ can be replaced by an identical $r$-graph on vertex set $V(G) \times \{1\}$, and $H$ by an identical $r$-graph on vertex set $V(H) \times \{2\}$, the definition of $G \dot\cup H$ extends naturally to pairs of $r$-graphs which are not necessarily disjoint.

If $F$ and $G$ are $r$-graphs, a map $f: V(F) \rightarrow V(G)$ is a \textit{homomorphism} if it maps edges to edges. An \textit{embedding} is an injective homomorphism. We shall frequently abuse the notion of subgraph and say $F$ is a subgraph of $G$ if there exists an embedding of $F$ into $G$. We will denote this by $F \subseteq G$, and if no confusion can arise we may identify $F$ with the image of its embedding in $G$.

If $\mathcal{F}$ is a family of $r$-graphs, the \textit{closure of $\mathcal{F}$ under homomorphisms} is the family $\overline{\mathcal{F}}$ containing all $r$-graphs $G$ for which there exists $F \in \mathcal{F}$ and a surjective homomorphism $f : V(F) \rightarrow V(G)$ (here $f$ is surjective on $V(G)$, but $G$ may contain edges not in the image of $f$). If $\mathcal{F} = \overline{\mathcal{F}}$ then $\mathcal{F}$ is \textit{closed under homomorphisms}. If for any $G \in \overline{\mathcal{F}}$ there exists $F \subseteq G$ with $F \in \mathcal{F}$ we say $\mathcal{F}$ is \textit{weakly closed under homomorphisms}.

If $\mathcal{F}$ is a family of $r$-graphs, an $r$-graph $G$ is \textit{$\mathcal{F}$-free} if no subgraph of $G$ belongs to $\mathcal{F}$. 

We define the $(n-1)$-dimensional simplex as
\begin{equation*}
\Delta_n = \{ (x_1, \ldots, x_n) \in \mathbb{R}^n : \sum_{i=1}^n x_i =1, x_i \geq 0\}.
\end{equation*}

\section{\normalsize The Infinity Principle}

If $G$ is an $r$-multigraph on $[n]$, we define a polynomial $p_G(x)$ as follows:
\begin{equation*}
p_G(x_1, \ldots, x_n) := r!\sum_{D \in E(G)} \prod_{i=1}^n \frac{x_i^{D(i)}}{D(i)!}.
\end{equation*}
The \textit{Lagrangian} of $G$ is defined to be
\begin{equation}
\lambda(G) := \max\,\{p_G(\mathbf{x}) : \mathbf{x} \in \Delta_n \}.
\end{equation}
The maximum is attained as it is taken over a compact set and $p_G$ is continuous.
An element $\mathbf{x} \in \Delta_n$ such that $p_G(\mathbf{x}) = \lambda(G)$ is an \textit{optimal vector for $G$}. Note that $\lambda(G)=0$ implies that $G$ has no edges.
For technical reasons we also define $\lambda(\emptyset) = 0$.

For $r \geq 1$, let $\Lambda^{(r)}$ be the set of values $\lambda(G)$, with $G$ an $r$-graph. Note that we do not take into account non-simple $r$-multigraphs. Pikhurko proved the following.

\begin{theorem}[Pikhurko, \cite{pikhurko12}]
\label{thm:pikfinite}
$\Lambda^{(r)} \subseteq \Pi_{\fin}^{(r)}$.
\end{theorem}

The weaker statement $\Lambda^{(r)} \subseteq \Pi_{\infty}^{(r)}$ is much simpler to prove. In particular if $e$ is an $r$-edge then $\lambda(e) = \frac{r!}{r^r} \in \Pi_{\infty}^{(r)}$. It was shown by Brown and Simonovits \cite{Brown84} that $\Lambda^{(r)}$ is dense in $\Pi_{\infty}^{(r)}$. As the latter is a closed set, this in fact proves the following.
\begin{lemma}
\label{lem:lambdaclos}
$\overline{\Lambda}^{(r)} = \Pi_{\infty}^{(r)}$.
\end{lemma}
We shall frequently rely on Lemma \ref{lem:lambdaclos} to transfer statements about $\Lambda^{(r)}$ to the whole of $\Pi_{\infty}^{(r)}$ via continuity.

Pikhurko further proved that $\lambda(G) \in \Pi_{\fin}^{(r)}$ for any $r$-multigraph $G$. We shall only need the following weaker statement.

\begin{lemma}
\label{lem:pattern}
For any $r$-multigraph $G$ we have $\lambda(G) \in \Pi_{\infty}^{(r)}$.
\end{lemma}

As Pikhurko's proof is long and difficult, we include here a short proof of Lemma \ref{lem:pattern}.

First we need a definition introduced by Pikhurko in \cite{pikhurko12}. We reproduce it here in a simplified variant that better suits our needs.

Let $G = (S, E)$ be an $r$-multigraph. Identify $S$ with $[m]$ and let $V_1, \ldots , V_m$ be disjoint sets with $V := V_1 \cup \ldots \cup V_m$. The \textit{profile} of an $r$-set $X \subseteq V$ (with respect to $V_1, \ldots , V_m$) is the $r$-multiset on $[m]$ that contains $i \in [m]$ with multiplicity $|X \cap V_i|$. For an $r$-multiset $Y \subseteq [m]$ let $Y((V_1, \ldots , V_m))$ consist of all $r$-subsets of $V$ whose profile is $Y$. We call this $r$-graph the \textit{blow-up} of $Y$ and the $r$-graph
\begin{equation*}
E((V_1 , \ldots , V_m)) := \bigcup_{Y \in E} Y((V_1 , \ldots , V_m ))
\end{equation*}
is called the \textit{blow-up of $E$} (with respect to $V_1, \ldots, V_m$). If all sets $V_i$ have the same size $t$, we denote $E((V_1, \ldots, V_m))$ by $G(t)$.

A $G$-construction on a set $V$ is any $r$-graph $E((V_1, \ldots ,V_m))$ obtained by taking a partition $V = V_1 \cup \ldots \cup V_m$. Let $p_n$ be the maximum number of edges of a $G$-construction on $n$ vertices. Then Pikhurko defined
\begin{equation*}
\Lambda_G := \lim_{n \rightarrow \infty} \frac{p_n}{\binom{n}{r}},
\end{equation*}
and proved that this limit always exists. It is easy to see that $\Lambda_G = \lambda(G)$. In fact Pikhurko defined a much larger class of $G$-constructions, where one is allowed to recursively apply the construction into some of the parts; this was a key step in his proof of Theorem \ref{thm:pikirrat}.

The main observation is now the following, which is implicit in \cite{FranklRodl84} and \cite{Brown84}.

\begin{lemma}[The Infinity Principle]
Let $\{G_n\}_{n \geq 1}$ be a sequence of $r$-graphs with $v(G_n) = n$ and $d(G_n) \rightarrow \alpha$. Suppose that for any sequence of $r$-graphs $H_n$ with $H_n \subseteq G_n$ and $v(H_n)$ tending to infinity, we have $\limsup_{n \rightarrow \infty} d(H_n) \leq \alpha$. Then $\alpha \in \Pi_{\infty}^{(r)}$.
\end{lemma}
\begin{proof}
Define $\mathcal{F}_\infty := \{H : H \not\subseteq G_n \textrm{ for any $n \geq 1$}\}$. We claim $\pi(\mathcal{F_\infty}) = \alpha$.

Indeed, for any $m \geq 1$, let $T_m$ be a maximum $\mathcal{F}_\infty$-free $r$-graph on $m$ vertices. Then for each $m \geq 1$, $T_m \notin \mathcal{F_\infty}$, and hence $T_m \subseteq G_{n(m)}$, for some $n(m)$ depending on $m$. Define $\{H_n\}_{n \geq 1}$ in the following way. If there exists $m$ with $n(m) = n$, let $H_n := T_m$ (if several choices exists, choose one with maximum density). Otherwise let $H_n := G_n$. Then $\lim_{m \rightarrow \infty} d(T_m) \leq \limsup_{n \rightarrow \infty} d(H_n) \leq \alpha$, by assumption. Hence $\pi(\mathcal{F_\infty}) \leq \alpha$. However, by construction $\pi(\mathcal{F_\infty}) \geq \alpha$, and so equality holds.
\end{proof}

\begin{proof}[Proof of Lemma \ref{lem:pattern}]
Let $G_n$ be a maximum\footnote{i.e. with a maximum number of edges.} $G$-construction on $n$ vertices. Let $H_n \subseteq G_n$ be any sequence of subgraphs with number of vertices tending to infinity. W.l.o.g. we may assume that $H_n$ is an induced subgraph on $m(n)$ vertices. Then $H_n$ is by definition also a $G$-construction, and hence has no more than $p_{m(n)}$ edges. Thus by definition of $\Lambda_G$, we must have $\limsup d(H_n) \leq \Lambda_G$. Consequently by the Infinity Principle, $\lambda(G) = \Lambda_G \in \Pi_{\infty}^{(r)}$.
\end{proof}

If $G = (V, E)$ is an $r$-multigraph, we define $\overline{G} = (V, V^{(r)} \setminus E)$. One of the advantages of working with multigraphs is the following.

\begin{lemma}
\label{lem:duality}
For any $r$-multigraph $G$ on $[n]$ and any $\mathbf{x} \in \Delta_n$ we have
\begin{equation}
\label{eq:duality}
p_G(\mathbf{x}) + p_{\overline{G}}(\mathbf{x}) = 1.
\end{equation}
\end{lemma}
\begin{proof}
Note that $p_G(\mathbf{x}) + p_{\overline{G}}(\mathbf{x}) = (\sum_{i=1}^n x_i)^r = 1$.
\end{proof}

\section{\normalsize The global structure}

In this section we prove Theorem \ref{thm:global}.

It is an easy exercise to check that $*$ is commutative, associative and cancellative. Furthermore the unit is the element $(1, 0)$, under the convention $0^0 = 1$. Thus we only need to show that $\Pi_{\infty}$ is closed under $*$. To this end we make the following definition.

\begin{definition}
Let $r, s \geq 0$, $G$ be an $r$-graph and $H$ an $s$-graph on disjoint vertex sets. We define $G * H$ as the $(r+s)$-graph on vertex set $V(G) \cup V(H)$ and edge set $\{e \cup f : e \in E(G), f \in E(H)\}$.
\end{definition}

This definition was introduced by Emtander \cite{emtander2008} in connection with Betti numbers of hypergraphs. It was also considered by Bollob\'as, Leader and Malvenuto in the context of Tur\'an densities \cite{Bollobas11}. The definition of $G * H$ extends naturally to any two (not necessarily disjoint) uniform hypergraphs $G$ and $H$.

\begin{proposition}
\label{prop:max}
Let $r, s \geq 1$ and $\mathfrak{f}:[0,1] \rightarrow \mathbb{R}$ be given by $\mathfrak{f}(x) = x^r(1-x)^s$. Then $\mathfrak{f}$ has a unique maximum $x_0 := \frac{r}{r+s}$ and furthermore $\mathfrak{f}(x_0) = \frac{r^rs^s}{(r+s)^{r+s}}$.
\end{proposition}
\begin{proof}
We see that $\mathfrak{f}'(x) = rx^{r-1}(1-x)^s - sx^r(1-x)^{s-1}$, and so $\mathfrak{f}'(x) = 0$ only happens for $x_0 := \frac{r}{r+s}$. As $\mathfrak{f}(0) = 0$, $\mathfrak{f}(x_0)$ must be a maximum point, and the claim follows.
\end{proof}

\begin{lemma}
Let $G$ be an $r$-graph and $H$ be an $s$-graph. Then $\lambda(G*H) = \lambda(G)\lambda(H)\binom{r+s}{r}\frac{r^rs^s}{(r+s)^{r+s}}$.
\end{lemma}
\begin{proof}
We assume w.l.o.g. that $G$ has vertex set $\{1, \ldots, n\}$ and $H$ has vertex set $\{n+1, \ldots, n+m\}$. Then $G*H$ has vertex set $[n+m]$.

By definition
\begin{equation*}
p_{G*H}(x_1,\ldots,x_n,y_1,\ldots,y_m) = \binom{r+s}{r}p_G(x_1,\ldots,x_n)p_H(y_1,\ldots,y_m).
\end{equation*}

Let $\mathbf{a} \in \Delta_n$ be an optimal vector for $G$ and $\mathbf{b} \in \Delta_m$ be an optimal vector for $H$. Let $\theta := \frac{r}{r+s}$. Then
\begin{align}
\lambda(G*H) &\geq p_{G*H}(\theta a_1,\ldots,\theta a_n, (1-\theta)b_1,\ldots,(1-\theta)b_m) \nonumber\\
&= \binom{r+s}{r}\theta^r\lambda(G)(1-\theta)^s\lambda(H) \nonumber\\
&=\lambda(G)\lambda(H)\binom{r+s}{r}\frac{r^rs^s}{(r+s)^{r+s}} \label{eq:lambda1}.
\end{align}

On the other hand, let $\mathbf{z} \in \Delta_{n+m}$ be an optimal vector for $G * H$. Set $M := \sum_{i=1}^n z_i$. Then
\begin{align*}
\lambda(G*H) &= p_{G*H}(z_1, \ldots, z_{n+m}) \\
&= \binom{r+s}{r}M^rp_G\left(\frac{z_1}{M}, \ldots, \frac{z_n}{M}\right)(1-M)^sp_H\left(\frac{z_{n+1}}{1-M}, \ldots, \frac{z_{n+m}}{1-M}\right) \\
&\leq \binom{r+s}{r}\lambda(G)\lambda(H)\mathfrak{f}(M) \\
&\leq \lambda(G)\lambda(H)\binom{r+s}{r}\frac{r^rs^s}{(r+s)^{r+s}}, \textrm{by Proposition \ref{prop:max}}.
\end{align*}
Together with \eqref{eq:lambda1} this proves the claim.
\end{proof}

\begin{proof}[Proof of Theorem \ref{thm:global}]
Let $(\alpha, r), (\beta, s) \in \Pi_{\infty}$. We want to show that $(\alpha, r) * (\beta, s) \in \Pi_{\infty}$.

We may assume that $r, s \geq 1$.

By Lemma \ref{lem:lambdaclos}, there exists a sequence of $r$-graphs $G_n$ with $\lambda(G_n) \rightarrow \alpha$. Similarly there exists a sequence of $s$-graphs $H_n$ with $\lambda(H_n) \rightarrow \beta$.
Then
\begin{align*}
\lim_{n\rightarrow \infty} \lambda(G_n * H_n) &= \lim_{n \rightarrow \infty} \left(\lambda(G_n)\lambda(H_n)\binom{r+s}{r}\frac{r^rs^s}{(r+s)^{r+s}}\right) \\
&= \alpha\beta\binom{r+s}{r}\frac{r^rs^s}{(r+s)^{r+s}}.
\end{align*}
Thus $\alpha\beta\binom{r+s}{r}\frac{r^rs^s}{(r+s)^{r+s}} \in \overline{\Lambda}^{(r+s)}$. Then Lemma \ref{lem:lambdaclos} completes the proof.
\end{proof}

\section{\normalsize The local structure}

In this section we prove Theorem \ref{thm:local}. The proof is naturally divided into two parts. We first prove the semigroup structure of $\Pi_{\infty}^{(r)}$ and then the closure of $\Pi_{\fin}^{(r)}$ under $\oplus_r$.

\subsection{\normalsize The semigroup structure}

Let $G$ and $H$ be two $r$-graphs on disjoint vertex sets. We define $G \oplus_r H$ as the $r$-multigraph with vertex set $V(G) \dot\cup V(H)$ and edge set 
\begin{equation*}
E(G \oplus_r H) = E(G) \dot\cup E(H) \dot\cup \{e \in (V(G) \cup V(H))^{(r)} : e \textrm{ \upshape{intersects both $V(G)$ and $V(H)$}}\}.
\end{equation*}
We then extend this definition to pairs of $r$-graphs with intersecting vertex sets in the same manner as before.

\begin{lemma}
\label{lem:max2}
Let $r \geq 2$ and $\alpha, \beta \in [0, 1]$. Define $\mathfrak{g}_{\alpha, \beta} : [0,1] \rightarrow \mathbb{R}$ by
\begin{equation*}
\mathfrak{g}_{\alpha, \beta}(x) = \alpha x^r + \beta (1-x)^r + r!\sum_{i=1}^{r-1}\frac{x^i(1-x)^{r-i}}{i!(r-i)!}.
\end{equation*}
 Then $\alpha \oplus_r \beta = \sup_{x \in [0,1]} \mathfrak{g}_{\alpha, \beta}(x)$, and moreover for $(\alpha, \beta) \neq (1,1)$, $\mathfrak{g}_{\alpha, \beta}$ is strictly concave and has a unique maximum at $x_{\alpha, \beta} := \frac{\sqrt[r-1]{1-\beta}}{\sqrt[r-1]{1-\alpha} + \sqrt[r-1]{1-\beta}}$.
\end{lemma}
\begin{proof}
Note that for $x \in [0,1]$ we have that
\begin{align*}
\alpha x^r + \beta (1-x)^r + r!\sum_{i=1}^{r-1}\frac{x^i(1-x)^{r-i}}{i!(r-i)!} &= \alpha x^r + \beta (1-x)^r + (1-x^r - (1-x)^r)\\
&= 1 - (1-\alpha)x^r - (1-\beta)(1-x)^r.
\end{align*}
Hence $\mathfrak{g}_{\alpha, \beta}(x) = 1 - (1-\alpha)x^r - (1-\beta)(1-x)^r$. If $\alpha = 1$ and $\beta = 1$ then $\sup_{x \in [0,1]} g(x) = 1 = \alpha \oplus_r \beta$ by definition. Hence we may assume that $\alpha < 1$ or $\beta < 1$. Then
\begin{equation*}
\mathfrak{g}_{\alpha, \beta}'(x) = -r(1-\alpha)x^{r-1} + r(1-\beta)(1-x)^{r-1}
\end{equation*}
and
\begin{equation*}
\mathfrak{g}_{\alpha, \beta}''(x) = -r(r-1)(1-\alpha)x^{r-2} - r(r-1)(1-\beta)(1-x)^{r-2}.
\end{equation*}
Thus $\mathfrak{g}_{\alpha, \beta}'' < 0$ on $(0,1)$, showing that $\mathfrak{g}_{\alpha, \beta}$ is strictly concave. Furthermore $\mathfrak{g}_{\alpha, \beta}'(x) = 0$ has a unique solution
\begin{equation*}
x_{\alpha, \beta} = \frac{\sqrt[r-1]{1-\beta}}{\sqrt[r-1]{1-\alpha} + \sqrt[r-1]{1-\beta}}.
\end{equation*}
Hence $\mathfrak{g}_{\alpha, \beta}(x)$ has the global maximum
\begin{equation*}
\mathfrak{g}_{\alpha, \beta}(x_{\alpha, \beta}) = 1 - \frac{(1-\alpha)(1-\beta)}{(\sqrt[r-1]{1-\alpha} + \sqrt[r-1]{1- \beta})^{r-1}},
\end{equation*}
which is the same as $\alpha \oplus_r \beta$.
\end{proof}

\begin{lemma}
\label{lem:addition}
Let $G$ and $H$ be two $r$-graphs. Then
\begin{equation*}
\lambda(G \oplus_r H) = \lambda(G) \oplus_r \lambda(H).
\end{equation*}
\end{lemma}
\begin{proof}
We shall assume that $G$ has vertex set $[n]$ and $H$ has vertex set $\{n+1, \ldots, n+m\}$. Then $G\oplus_r H$ has vertex set $[n+m]$.

Let $\mathbf{x} \in \Delta_{n+m}$ arbitrary. Set $S_x := \sum_{i=1}^n x_i$ and note that
\begin{align*}
p_{G \oplus_r H}(x_1, \ldots, x_{n+m}) =\ &p_G(\frac{x_1}{S_x}, \ldots, \frac{x_n}{S_x})S_x^r + 
p_H(\frac{x_{n+1}}{1-S_x}, \ldots, \frac{x_{n+m}}{1-S_x})(1-S_x)^r \\
&+ r!\sum_{i=1}^{r-1}\frac{S_x^i(1-S_x)^{r-i}}{i!(r-i)!}.
\end{align*}
Moreover $(\frac{a_1}{S_x}, \ldots, \frac{a_n}{S_x}) \in \Delta_n$ and $(\frac{a_{n+1}}{1-S_x}, \ldots, \frac{a_{n+m}}{1-S_x}) \in \Delta_m$. Hence
\begin{align*}
\lambda(G \oplus_r H) &= \sup_{x \in [0,1]} \left\{\lambda(G)x^r + \lambda(H)(1-x)^r + r!\sum_{i=1}^{r-1}\frac{x^i(1-x)^{r-i}}{i!(r-i)!}\right\}\\
&= \lambda(G) \oplus_r \lambda(H),
\end{align*}
by Lemma \ref{lem:max2}. This proves the lemma.
\end{proof}

We can now prove the following.

\begin{lemma}
\label{lem:local_inf}
$(\Pi_{\infty}^{(r)}, \oplus_r)$ is a commutative topological semigroup.
\end{lemma}
\begin{proof}
Commutativity and associativity are simple exercises left to the reader. Continuity of $\oplus_r$ is clear everywhere except at $(1, 1)$.

Let $\{x_n\}_{n \geq 1}$ and $\{y_n\}_{n \geq 1}$ be arbitrary sequences of real numbers from $[0,1)$ converging to $1$. Then by definition $x_n \oplus_r y_n \leq 1$, and we want to show that equality holds in the limit. Let $\eps > 0$ be arbitrary and define $\delta := 2^{r-1}\eps$. Then there exists an $n_0 \geq 1$ such that $1 - x_n < \delta$ and $1 - y_n < \delta$. Let $n \geq n_0$ and assume w.l.o.g. that $x_n \geq y_n$. Then
\begin{equation*}
1 - \frac{(1-x_n)(1-y_n)}{(\sqrt[r-1]{1 - x_n} + \sqrt[r-1]{1 - y_n})^{r-1}} \geq 1 - \frac{1-y_n}{2^{r-1}} > 1 - \frac{\delta}{2^{r-1}} = 1 - \eps.
\end{equation*}
Hence $\oplus_r$ is continuous at $(1, 1)$ as well.

Thus we only need to prove that $\Pi_{\infty}^{(r)}$ is closed with respect to $\oplus_r$.

Let $\alpha, \beta \in \Pi_{\infty}^{(r)}$. By Lemma \ref{lem:lambdaclos}, we can choose a sequence of $r$-graphs $G_n$ with $\lambda(G_n) \rightarrow \alpha$, and a sequence of $r$-graphs $H_n$ with $\lambda(H_n) \rightarrow \beta$. 

Consider the sequence $G_n \oplus_r H_n$. By Lemma \ref{lem:addition} and continuity of $\oplus_r$,
\begin{align*}
\lim_{n \rightarrow \infty} \lambda(G_n \oplus_r H_n) &= 
\lim_{n \rightarrow \infty} \lambda(G_n) \oplus_r \lambda(H_n) \\
&= \alpha \oplus_r \beta.
\end{align*}
By Lemma \ref{lem:pattern}, $\lambda(G_n \oplus_r H_n) \in \Pi_{\infty}^{(r)}$. As this is a closed set, $\alpha \oplus_r \beta \in \Pi_{\infty}^{(r)}$, proving the lemma.
\end{proof}

We are now left to prove that $\Pi_{\fin}^{(r)}$ is closed under $\oplus_r$. This task will be substantially more difficult.

\subsection{\normalsize The $\lambda$ function}

For the remaining part of the proof we shall need a number of additional statements, which we gather in the following $3$ subsections.

We start with several observations on the Lagrangian function.

\begin{lemma}
\label{lem:lambdaprop}
For any $r \geq 2$ the following holds.
\renewcommand{\theenumi}{(\roman{enumi})}
\begin{enumerate}
\item If $H \subseteq G$ are $r$-graphs then $\lambda(H) \leq \lambda(G)$.
\item If $f: G \rightarrow H$ is a homomorphism of $r$-graphs then $\lambda(H) \geq \lambda(G)$.
\item If $G$ and $H$ are $r$-graphs then $\lambda(G \dot\cup H) = \max\{\lambda(G), \lambda(H)\}$.
\end{enumerate}
\begin{proof}
Statement (i) is clear.

We prove (ii). Assume $G$ has vertex set $[n]$ and $H$ has vertex set $[m]$. Let $\mathbf{a} \in \Delta_n$ be an optimal vector for $G$. Define $\mathbf{b} \in \Delta_m$ by setting $b_i = \sum_{j \in f^{-1}(i)} a_j$. As $f$ maps edges to edges, it follows that 
\begin{align*}
\lambda(H) \geq p_H(\mathbf{b}) &= r!\sum_{e' \in E(H)} \prod_{i \in e'} (\sum_{j \in f^{-1}(i)} a_j)\\
&\geq r!\sum_{e' \in E(H)} \sum_{e \in f^{-1}(e')} \prod_{j \in e} a_j = r!\sum_{e \in E(G)} \prod_{j \in e} a_{j} = p_G(\mathbf{a}) = \lambda(G).
\end{align*}
This proves (ii).

We prove (iii). Assume $G$ has vertex set $[n]$ and $H$ has vertex set $\{n+1, \ldots, n+m\}$. For any $\mathbf{x} \in \Delta_{n+m}$ let $S_x := \sum_{i=1}^n x_i$. Then
\begin{align*}
p_{G \dot\cup H}(\mathbf{x}) &= S_x^r p_G(\frac{x_1}{S_x}, \ldots, \frac{x_n}{S_x}) + (1-S_x)^rp_H(\frac{x_{n+1}}{1-S_x}, \ldots, \frac{x_{n+m}}{1-S_x})\\
&= S_x^r \lambda(G) + (1-S_x)^r \lambda(H).
\end{align*}
Consider the function $f:[0,1] \rightarrow [0,1]$ given by $f(x) = x^r\lambda(G) + (1-x)^r\lambda(H)$. Then the second derivative $f''(x) \geq 0$, hence $f$ is convex. So the maximum is achieved at one of the endpoints of the interval. This implies that $\lambda(G \dot\cup H) = \max\{\lambda(G), \lambda(H)\}$, proving (iii). 
\end{proof}
\end{lemma}

\subsection{\normalsize The $\pi$ function}

We gather in this subsection several results about the $\pi$ function.

\begin{theorem}[Theorem 2, \cite{Brown84}]
\label{thm:supers}
For any $\eps > 0$ and any family $\mathcal{F}$ of $r$-graphs, there are $\delta > 0$ and $n_S$ such that the following holds. Any $r$-graph $G$ on $n \geq n_S$ vertices and with more than $(\pi(\mathcal{F})+\eps)\binom{n}{r}$ edges contains at least $\delta n^{v(F)}$ copies of some $F \in \mathcal{F}$.
\end{theorem}

Theorem \ref{thm:supers} is a generalization of the supersaturation theorem of Erd\H os and Simonovits. It has the following consequence (see the proof of Theorem $2.2$ in \cite{keevash11}).

\begin{lemma}
\label{lem:blowup}
For any $t \geq 1$, $\eta > 0$ and any $r$-graph $F$ there exist $\rho > 0$ and $n_L$ such that the following holds. If $G$ is an $r$-graph on $n \geq n_L$ vertices containing at least $\eta n^{v(F)}$ copies of $F$, then $G$ contains at least $\rho n^{v(F)t}$ copies of $F(t)$.
\end{lemma}

\begin{lemma}
\label{lem:remset}
For any $\eta > 0$, $a \geq 1$ and $k \geq 1$ there is an $n_C$ such that the following holds. If $G$ is an $r$-graph on $n \geq n_C$ vertices containing at least $\eta n^{v(F)}$ copies of some $r$-graph $F$ on at most $k$ vertices and $A \subset V(G)$ is a set of size $a$, then $G \setminus A$ contains at least $\frac{\eta}{2}n^{v(F)}$ copies of $F$.
\end{lemma}
\begin{proof}
Set $n_C := \frac{2ak}{\eta}$. We show that the claim holds for $n_C$.

The number of copies of $F$ intersecting $A$ is at most $v(F)an^{v(F)-1} \leq kan^{v(F)-1}$. Hence the number of copies of $F$ disjoint from $A$ is at least
\begin{equation*}
\eta n^{v(F)} - kan^{v(F)-1} \geq \frac{\eta n^{v(F)}}{2},
\end{equation*}
which holds by our choice of $n_C$.
\end{proof}

The $\pi$ function has several properties, which we list below.

\begin{lemma}
\label{lem:piprop}
For any $r \geq 2$ the following holds.
\renewcommand{\theenumi}{(\roman{enumi})}
\begin{enumerate}
\item If $\mathcal{F}$ is a family of $r$-graphs and $\mathcal{F} \subseteq \mathcal{F}'$ then $\pi(\mathcal{F}') \leq \pi(\mathcal{F})$.
\item If $H \subseteq G$ are two $r$-graphs and $\mathcal{F}$ is a family of $r$-graphs then $\pi(\mathcal{F} \cup \{H\}) \leq \pi(\mathcal{F} \cup \{G\})$.
\item If $G$ and $H$ are $r$-graphs and $\mathcal{F}$ is a family of $r$-graphs then
\begin{equation}
\label{eq:disjunion}
\pi(\mathcal{F} \cup \{G \dot\cup H\}) = \max\{ \pi(\mathcal{F} \cup \{H\}), \pi(\mathcal{F} \cup \{G\})\}.
\end{equation}
\item If $F$ is any $r$-graph, $t \geq 1$ and $\mathcal{F}$ is a family of $r$-graphs then $\pi(\mathcal{F} \cup \{F\}) = \pi(\mathcal{F} \cup \{F(t)\})$.
\item If $\mathcal{F}$ is a family of $r$-graphs then $\pi(\mathcal{F}) = 1$ if and only if $\mathcal{F}$ is empty.
\end{enumerate}
\end{lemma}
\begin{proof}
The statements (i) and (ii) are clear.

We prove (iii). Assume w.l.o.g. that $\pi(\mathcal{F} \cup \{G\}) \geq \pi(\mathcal{F} \cup \{H\})$. By (ii) we have that
\begin{equation*}
\pi(\mathcal{F} \cup \{G \dot\cup H\}) \geq \pi(\mathcal{F} \cup \{G\}).
\end{equation*}
Assume for a contradiction that this inequality is strict. Then there exists $\eps > 0$ such that $\pi(\mathcal{F} \cup \{G \dot\cup H\}) > \pi(\mathcal{F} \cup \{G\}) + \eps$. Let $\delta$ and $n_S$ be given by Theorem \ref{thm:supers} on input $\eps$ and $\mathcal{F} \cup \{G\}$. Let $n_C$ be given by Lemma \ref{lem:remset} on input $\delta$, $a := v(H)$ and $k := v(G)$. Furthermore, let $n_1$ be large enough so that for any $n \geq n_1, \ex(n, \mathcal{F} \cup \{G \dot\cup H\}) > (\pi(\mathcal{F} \cup \{G\}) + \eps)\binom{n}{r} > \ex(n, \mathcal{F} \cup \{H\})$.

Now let $n \geq \max\{n_S, n_C, n_1, \frac{2}{\delta}\}$. By assumption there must exist an $r$-graph $G'$ on $n$ vertices with $e(G') > (\pi(\mathcal{F} \cup \{G\}) + \eps)\binom{n}{r}$, which is $(\mathcal{F} \cup \{G \dot\cup H\})$-free. As $n \geq n_1$, there must exist a copy of $H$ in $G'$, which we now fix, and also denote by $H$. However, by our choice of $n$ and Theorem \ref{thm:supers}, there are at least $\delta n^{v(G)}$ copies of $G$ in $G'$. Consequently by Lemma \ref{lem:remset}, at least $\frac{\delta}{2}n^{v(G)}$ of them are disjoint from $H$, in particular we can find a copy of $G \dot\cup H$ in $G'$, a contradiction. This proves (iii).

We prove (iv). By (ii) we have that $\pi(\mathcal{F} \cup \{F\}) \leq \pi(\mathcal{F} \cup \{F(t)\})$. Assume for a contradiction that 
this inequality is strict. Then there exists $\eps > 0$ such that $\pi(\mathcal{F} \cup \{F(t)\}) > \pi(\mathcal{F} \cup \{F\}) + \eps$. 
Let $\delta$ and $n_S$ be given by Theorem \ref{thm:supers} on input $\eps$ and $\mathcal{F} \cup \{F\}$. 
Let $\rho$ and $n_L$ be given by Lemma \ref{lem:blowup} on input $t$, $\delta$ and $F$. Furthermore, let $n_1$ be large enough so that for any $n \geq n_1, \ex(n, \mathcal{F} \cup \{F(t)\}) > (\pi(\mathcal{F} \cup \{F\}) + \eps)\binom{n}{r}$.

Now let $n \geq \max\{n_S, n_L, n_1, \frac{1}{\rho}\}$. By assumption there must exist an $r$-graph $G$ on $n$ vertices with $e(G) > (\pi(\mathcal{F} \cup \{F\}) + \eps)\binom{n}{r}$, which is $(\mathcal{F} \cup \{F(t)\})$-free. By Theorem \ref{thm:supers}, there are at least $\delta n^{v(F)}$ copies 
of $F$ in $G$. Consequently by Lemma \ref{lem:blowup}, there is at least one copy of $F(t)$ in $G$, a contradiction. This proves (iv).

We prove (v). Clearly $\pi(\emptyset) = 1$, so assume $\mathcal{F}$ is a non-empty family of $r$-graphs. Let $F \in \mathcal{F}$. Then $\pi(\mathcal{F}) \leq \pi(F)$ by (i), and we claim that $\pi(F) < 1$. This intuitive claim can be proved in several ways, for example by using the following inequality of Sidorenko \cite{Sidorenko89}: if $F$ is an $r$-graph with $f \geq 2$ edges then $\pi(F) \leq \frac{f-2}{f-1}$. This proves (v).
\end{proof}

Finally, the following two lemmas will provide a better understanding of the structure of extremal $r$-graphs.

\begin{lemma}
\label{lem:min_degree}
Let $ r \geq 2$ and $\mathcal{F}$ be a family of $r$-graphs weakly closed under homomorphisms. Set $\alpha := \pi(\mathcal{F})$. For any $\delta > 0$ there exists an $n_D \geq 1$ such that any maximum $\mathcal{F}$-free $r$-graph on $n \geq n_D$ vertices has minimum degree at least $(\alpha - \delta)\binom{n-1}{r-1}$.
\end{lemma}
\begin{proof}
Choose $n_D \geq 1$ so that for any $n \geq n_D$ we have $\ex(n, \mathcal{F}) \geq (\alpha - \frac{\delta}{2})\binom{n}{r}$, and furthermore $n_D > 1 + \frac{2(r-1)}{\delta}$.

Let $G$ be any maximum $\mathcal{F}$-free $r$-graph on $n \geq n_D$ vertices, and assume for a contradiction that $G$ contains a vertex $x$ of degree $d(x) < (\alpha - \delta)\binom{n-1}{r-1}$.

As $e(G) = \sum_{x \in V(G)} \frac{d(x)}{r}$, there must exist a vertex $v \in V(G)$ of degree $d(v) \geq (\alpha - \frac{\delta}{2})\binom{n-1}{r-1}$. Then replace $x$ by a new vertex $v'$ and add all edges $\{v'\} \cup e$, where $x \notin e$ and $\{v\} \cup e \in E(G)$. In other words, we replace $x$ by a copy of $v$, but we duplicate only the edges incident with $v$ and not with $x$. Let $G'$ be the resulting $r$-graph. As $\mathcal{F}$ is weakly closed under homomorphisms, $G'$ is still $\mathcal{F}$-free.

However,
\begin{align*}
e(G') - e(G) &\geq (\alpha - \frac{\delta}{2})\binom{n-1}{r-1} - (\alpha - \delta)\binom{n-1}{r-1} - \binom{n-2}{r-2}\\
&=\left(\frac{\delta(n-1)}{2(r-1)} - 1\right)\binom{n-2}{r-2}\\
&> 0,
\end{align*}
as $n \geq n_D$. This contradicts the maximality of $G$.
\end{proof}
One can strengthen the proof of Lemma \ref{lem:min_degree} to show that in a maximum $\mathcal{F}$-free $r$-graph all degrees are roughly the same $(\alpha + o(1))\binom{n-1}{r-1}$. We shall not need this stronger statement, but rather a variation of it.
\begin{lemma}
\label{lem:add_vertex}
Let $r \geq 2$ and $\mathcal{F}$ be a family of $r$-graphs weakly closed under homomorphisms. Set $\alpha := \pi(\mathcal{F})$. For any $\eps > 0$, there exist a $\tau > 0$ and an $n_V \geq 1$ such that the following holds. If $G$ is any $r$-graph on $n \geq n_V$ vertices, density at least $\alpha - \tau$ and having a vertex $v$ of degree at least $(\alpha + \eps)\binom{n-1}{r-1}$, then $G$ is not $\mathcal{F}$-free.  
\end{lemma}
\begin{proof}
We assume w.l.o.g. that $\eps < 1$.

Set $\tau := \frac{\eps^2 r}{72(r-1)}$ and choose $n_V \geq 2$ so that for any $n \geq n_V$, $\ex(n, \mathcal{F}) < (\alpha + \tau)\binom{n}{r}$.

We shall assume for a contradiction that $G$ is $\mathcal{F}$-free. Set $n := v(G)$. Then
\begin{equation*}
(\alpha + \frac{\eps}{4})\binom{n}{r} > e(G) = \sum_{x \in V(G)} \frac{d(x)}{r}.
\end{equation*}
Let $S := \{x \in V(G) : d(x) \leq (\alpha + \frac{\eps}{2})\binom{n-1}{r-1}\}$. Then $e(G) \geq \frac{n - |S|}{r}(\alpha + \frac{\eps}{2})\binom{n-1}{r-1}$, so
\begin{equation*}
|S| \geq (1 - \frac{\alpha + \eps / 4}{\alpha + \eps/2})n = \frac{\eps}{4\alpha + 2\eps}n.
\end{equation*}
Note that $\frac{\eps}{4\alpha + 2\eps} > \frac{\eps}{6(r-1)}$, hence we can fix $S' \subseteq S$ of size $\frac{\eps n}{6(r-1)}$ (here and in what follows we ignore upper and lower integer parts; this does not affect our arguments).

We construct a new $r$-graph $G'$ from $G$ by deleting all edges incident to $S'$ and adding all edges $\{\{x\} \cup e : \{v\} \cup e \in E(G \setminus S'), x \in S'\}$. Then $S' \cup \{v\}$ is an independent set in $G'$. 

We claim $G'$ is $\mathcal{F}$-free. Indeed, if $f : V(F) \rightarrow V(G')$ is any embedding of a graph $F \in \mathcal{F}$ into $G'$, then composing $f$ with the map $g : V(G') \rightarrow V(G)$ that sends $S'$ to $v$ and is the identity otherwise, gives a homomorphism of $F$ into $G$. Thus there exists a surjective homomorphism $f' : F \rightarrow F'$ with $F' \subseteq G$. As $\mathcal{F}$ is weakly closed under homomorphisms, $F'$ and hence $G$ contains an element of $\mathcal{F}$ as a subgraph, a contradiction.

Consequently $e(G') < (\alpha + \tau)\binom{n}{r}$. But
\begin{align*}
e(G') - e(G) &\geq |S'| \left(d_G(v) - |S'|\binom{n-2}{r-2}\right) - |S'|(\alpha + \frac{\eps}{2})\binom{n-1}{r-1}\\
&\geq |S'| \left((\alpha + \eps)\binom{n-1}{r-1} - |S'|\binom{n-2}{r-2}\right) - |S'|(\alpha + \frac{\eps}{2})\binom{n-1}{r-1}\\
&\geq |S'|\binom{n-2}{r-2}((\alpha + \eps)\frac{n-1}{r-1} - |S'|) - |S'|(\alpha + \frac{\eps}{2})\binom{n-1}{r-1}\\
&\geq |S'|(\alpha + \frac{2\eps}{3})\binom{n-1}{r-1} - |S'|(\alpha + \frac{\eps}{2})\binom{n-1}{r-1}, \textrm{ as $|S'| \leq \frac{\eps n}{6(r-1)}$ and $n_V \geq 2$,}\\
&= \frac{\eps}{6}|S'|\binom{n-1}{r-1}\\
&=\frac{\eps^2 r}{36(r-1)}\binom{n}{r},
\end{align*}
which is at least $2\tau \binom{n}{r}$. Hence $e(G') \geq e(G) + 2\tau \binom{n}{r} \geq (\alpha + \tau)\binom{n}{r}$, a contradiction. 
\end{proof}

\subsection{\normalsize The $\oplus_r$ function}

We now study the map $\oplus_r : [0,1] \times [0,1] \rightarrow [0,1]$.

\begin{lemma}
\label{lem:monoton}
The $\oplus_r$ function is nondecreasing in each of its arguments on $[0,1] \times [0,1]$. In fact, for any $\alpha, \beta \in [0,1)$ and $0 \leq \eps \leq \alpha$ we have
\begin{equation}
\label{eq:increase}
\alpha \oplus_r \beta \geq (\alpha - \eps) \oplus_r \beta + \eps\left(\frac{\sqrt[r-1]{1-\beta}}{1+\sqrt[r-1]{1 - \beta}}\right)^r
\end{equation}
\end{lemma}
\begin{proof}
The first statement follows immediately from the definition of $\oplus_r$.

To prove the second, define $h_{\beta} : [0,1] \rightarrow \mathbb{R}$ by
\begin{equation*}
h_\beta(x) = 1 - \frac{(1-x)(1-\beta)}{(\sqrt[r-1]{1-x} + \sqrt[r-1]{1-\beta})^{r-1}}.
\end{equation*}
Then the first order derivative of $h_\beta(x)$ exists and it is equal to
\begin{equation*}
h_\beta'(x) = \left(\frac{\sqrt[r-1]{1-\beta}}{\sqrt[r-1]{1-x}+\sqrt[r-1]{1-\beta}}\right)^r. 
\end{equation*}
Thus for any $x \in [0,1)$ we have
\begin{equation*}
h_\beta'(x) \geq h_\beta'(0) = \left(\frac{\sqrt[r-1]{1-\beta}}{1+\sqrt[r-1]{1 - \beta}}\right)^r.
\end{equation*}
Hence
\begin{align*}
\alpha \oplus_r \beta &= (\alpha - \eps) \oplus_r \beta + \int_{\alpha - \eps}^\alpha h_\beta'(x)\,dx\\
&\geq (\alpha - \eps) \oplus_r \beta + \eps h_\beta'(0)\\
&= (\alpha - \eps) \oplus_r \beta + \eps\left(\frac{\sqrt[r-1]{1-\beta}}{1+\sqrt[r-1]{1 - \beta}}\right)^r,
\end{align*}
proving the lemma.
\end{proof}

\subsection{\normalsize The Rigidity Lemma}

From this point on we adopt the strategy developed by Pikhurko in \cite{pikhurko12} (which in turn follows the Stability Method pioneered by Simonovits). The first step is to prove a \textit{rigidity lemma}: we construct some graphs which can embed only in a prescribed way in a graph of the form $G \times H$, where $\times$ is a special type of product of hypergraphs which we now define.

Let $G$ and $H$ be two $r$-graphs on disjoint vertex sets. We define $G \times H$ as the $r$-graph with vertex set $V(G) \dot\cup V(H)$ and edge set $E(G) \dot\cup E(H) \dot\cup \{e \in \binom{V(G) \dot\cup V(H)}{r} : e \textrm{ \upshape{intersects both $V(G)$ and $V(H)$}}\}$. We then extend this definition to $r$-graphs with intersecting vertex sets in the same manner as before.

If $\mathcal{F}$ is family of $r$-graphs and $M \geq 1$ an integer, we let $\mathcal{F}(M)$ be a maximal family of $r$-graphs with the following properties:
\renewcommand{\theenumi}{(\roman{enumi})}
\begin{enumerate}
\item $\mathcal{F} \subseteq \mathcal{F}(M)$.
\item If $F \in \mathcal{F}(M) \setminus \mathcal{F}$ then $F$ has at least one edge and $v(F) \leq M$.
\item $\pi(\mathcal{F}(M)) = \pi(\mathcal{F})$.
\end{enumerate}
We call $\mathcal{F}(M)$ \textit{an $M$-closure of $\mathcal{F}$}. Clearly there could be several distinct $M$-closures for a fixed $\mathcal{F}$, and moreover an $M$-closure always exists for any $M \geq 1$.

Suppose now that $\pi(\mathcal{F}) > 0$. By maximality, for any $F \notin \mathcal{F}(M)$ on at most $M$ vertices, we have $\pi(\mathcal{F}(M) \cup \{F\}) < \pi(\mathcal{F})$. Thus we can define the \textit{threshold of $\mathcal{F}(M)$} as
\begin{equation}
\label{eq:thresh}
\theta(\mathcal{F}(M)) := \max\{\pi(\mathcal{F}(M) \cup \{F\}) : F \notin \mathcal{F}(M) \textrm{ and } v(F) \leq M\},
\end{equation}
and this number is well-defined and strictly less than $\pi(\mathcal{F})$. For any $0 < \eps < \pi(\mathcal{F}) - \theta(\mathcal{F}(M))$ and any $F \notin \mathcal{F}(M)$ on at most $M$ vertices, Theorem \ref{thm:supers} applied to $\eps$ and $\mathcal{F}(M) \cup \{F\}$ gives us a $\delta(\eps, F) > 0$ and an $n_S(\eps, F) \geq 1$. We set
\begin{align*}
\delta(\eps, \mathcal{F}(M)) &:= \min\{ \delta(\eps, F) : F \notin \mathcal{F}(M) \textrm{ and } v(F) \leq M\},\\
n^*(\eps, \mathcal{F}(M)) &:= \max\{ n_S(\eps, F) : F \notin \mathcal{F}(M) \textrm{ and } v(F) \leq M\}.
\end{align*}

If $\pi(\mathcal{F}) = 0$ then $\mathcal{F}(M)$ contains all $r$-graphs on at most $M$ vertices and with at least one edge. For technical reasons we define $\theta(\mathcal{F}(M)) = -1$, and for any $0 < \eps < 1$, we set $\delta(\eps, \mathcal{F}(M)) = 1$ and $n^{*}(\eps, \mathcal{F}(M)) = 1$.

If $\mathcal{F}$ is family of $r$-graphs and $F$ is any $r$-graph, we say $F$ is \textit{valid with respect to (w.r.t) $\mathcal{F}$} if $\pi(\mathcal{F} \cup \{F\}) < \pi(\mathcal{F})$ or $\lambda(F) = \pi(\mathcal{F}) = 0$ (the point of this last condition is that we want the $1$-vertex graph to be valid w.r.t. any family of $r$-graphs). Otherwise we call $F$ \textit{invalid w.r.t. $\mathcal{F}$}. We say $F$ is \textit{minimal invalid w.r.t. $\mathcal{F}$} if $F$ is not valid, but for any $x \in V(F)$, the $r$-graph $F \setminus x$ is valid. Note that any invalid $r$-graph contains at least one edge. Moreover the empty graph is valid w.r.t. $\mathcal{F}$ for any family of $r$-graphs $\mathcal{F}$.

One particular example the reader should keep in mind is the family $\mathcal{F} = \{I_{r-1}\}$, where $I_{r-1}$ is the $r$-graph consisting of $r-1$ isolated vertices. By our definition $I_{r-1}$ is valid w.r.t. $\mathcal{F}$. To complicate matters further, any $I_{r-1}$-free $r$-graph has a bounded number of vertices and hence there is no sequence of extremal graphs with size tending to infinity. Later on we will show that we can avoid working with such families, but we will allow them for now.

Now suppose two families of $r$-graphs $\mathcal{F}_\alpha$ and $\mathcal{F}_\beta$ are given. Let $F$ be any $r$-graph.
A partition of $V(F)$ into $C_1$ and $C_2$ is denoted by $(C_1, C_2)$, and we identify $C_1$ with the $r$-graph $F[C_1]$,
and similarly $C_2$ with the $r$-graph $F[C_2]$. We allow $C_1$ or $C_2$ to be empty. A partition $(C_1, C_2)$ is called
\textit{valid w.r.t $\mathcal{F}_\alpha$ and $\mathcal{F}_\beta$}
if $C_1$ is valid w.r.t. $\mathcal{F}_\alpha$ and $C_2$ is valid w.r.t. $\mathcal{F}_\beta$.

Let us now record several simple observations concerning valid graphs.

\begin{lemma}
\label{lem:valid}
Let $\mathcal{F}$ be any family of $r$-graphs, $M \geq 1$ arbitrary and $\mathcal{F}(M)$ an arbitrary $M$-closure of $\mathcal{F}$. Then the following holds.
\begin{enumerate}
\item If $H \subseteq G$ and $G$ is valid w.r.t. $\mathcal{F}$ then so is $H$.
\item If $H \subseteq G$ and $H$ is invalid w.r.t. $\mathcal{F}$ then so is $G$.
\item If $G$ and $H$ are both valid w.r.t. $\mathcal{F}$ then so is $G \dot\cup H$.
\item If $G$ is valid w.r.t. $\mathcal{F}$ then $G$ is also valid w.r.t. $\mathcal{F}(M)$.
\item If $G \in \mathcal{F}$ is invalid w.r.t. $\mathcal{F}$ then $G$ is also invalid w.r.t. any family of $r$-graphs $\mathcal{G}$ containing $\mathcal{F}$, in particular $G$ is invalid w.r.t. $\mathcal{F}(M)$.
\end{enumerate}
\end{lemma}
\begin{proof}
Set $\alpha := \pi(\mathcal{F})$. We first prove (i).

If $\alpha > 0$ then by Lemma \ref{lem:piprop}, (ii), we have that $\pi(\mathcal{F} \cup \{H\}) \leq \pi(\mathcal{F} \cup \{G\}) < \alpha$, and hence $H$ is valid. If $\alpha = 0$, by Lemma \ref{lem:lambdaprop}, (i), we have that $\lambda(H) \leq \lambda(G) = 0$, and hence $H$ is again valid. Thus (i) holds.

We prove (ii).

By Lemma \ref{lem:piprop}, (i) and (ii), $\alpha = \pi(\mathcal{F} \cup \{H\}) \leq \pi(\mathcal{F} \cup \{G\}) \leq \alpha$ and so $\pi(\mathcal{F} \cup \{G\}) = \alpha$. Furthermore as $H$ is invalid we have $\lambda(G) \geq \lambda(H) > 0$. Hence (ii) holds as well.

We now prove (iii).

If $\alpha = 0$ then by Lemma \ref{lem:lambdaprop}, (iii),
\begin{equation*}
\lambda(G \dot\cup H) = \max\{\lambda(G), \lambda(H)\} = 0,
\end{equation*}
as $G$ and $H$ are both valid w.r.t. $\mathcal{F}$. Hence $G \dot\cup H$ is also valid.

If $\alpha > 0$, then by Lemma \ref{lem:piprop}, (iii),
\begin{equation*}
\pi(\mathcal{F} \cup \{G \dot\cup H\}) = \max\{\pi(\mathcal{F} \cup \{G\}), \pi(\mathcal{F} \cup \{H\})\} < \alpha,
\end{equation*}
again as $G$ and $H$ are both valid w.r.t. $\mathcal{F}$. Thus $G \dot\cup H$ is also valid, showing (iii).

We prove (iv).

Let $G$ be any $r$-graph valid w.r.t. $\mathcal{F}$. If $\alpha > 0$, then by Lemma \ref{lem:piprop}, (i),
\begin{equation*}
\pi(\mathcal{F}(M) \cup \{G\}) \leq \pi(\mathcal{F} \cup \{G\}) < \pi(\mathcal{F}) = \pi(\mathcal{F}(M)) 
\end{equation*}
and hence $G$ is valid w.r.t. $\mathcal{F}(M)$.

If $\alpha = 0$ then $\pi(\mathcal{F}(M)) = 0$ as well, and again $G$ is valid w.r.t. $\mathcal{F}(M)$. This shows (iv).

We prove (v).

Assume $G \in \mathcal{F}$ and $G$ is invalid w.r.t. $\mathcal{F}$. Then $\lambda(G) > 0$. If $\mathcal{G} \supseteq \mathcal{F}$ is any family of $r$-graphs containing $\mathcal{F}$, then $G \in \mathcal{G}$. Hence $\pi(\mathcal{G} \cup \{G\}) = \pi(\mathcal{G})$. Thus $G$ is invalid w.r.t. $\mathcal{G}$, showing (v).
\end{proof}

Note that in Lemma \ref{lem:valid}, (v), the assumption $G \in \mathcal{F}$ played a crucial role.

\begin{lemma}[The Rigidity Lemma]
\label{lem:rigid}
Let $\mathcal{F}_\alpha$ and $\mathcal{F}_\beta$ be two non-empty families of $r$-graphs with $\pi(\mathcal{F}_\alpha) = \alpha$ and $\pi(\mathcal{F}_\beta) = \beta$.

Let $P$ be any $r$-graph valid or minimal invalid w.r.t. $\mathcal{F}_\alpha$ such that if $P$ is minimal invalid, then $P \in \mathcal{F}_\alpha$. Let $Q$ be any $r$-graph minimal invalid w.r.t. $\mathcal{F}_\beta$ such that $Q \in \mathcal{F}_\beta$.

For any choice of $v \in V(P)$ and $w \in V(Q)$, there exists an $M_{P, Q, v, w} > 0$ such that for any $M \geq M_{P, Q, v, w}$ and any $M$-closures $\mathcal{F}_{\alpha}(M)$ and $\mathcal{F}_{\beta}(M)$ the following holds.

There exists an $r$-graph $C(P, Q, v, w)$ with the following properties:
\renewcommand{\theenumi}{(\Alph{enumi})}
\begin{enumerate}
\item $v(C(P, Q, v, w)) \leq M_{P, Q, v, w}$.
\item Let $K$ be the graph obtained from $P$ and $Q$ by identifying $v$ with $w$. Then $C(P, Q, v, w)$ contains an induced copy $K'$ of $K$ such that $C(P, Q, v, w) \setminus K'$ has a valid partition $(C_1, C_2)$ w.r.t. $\mathcal{F}_{\alpha}(M)$ and $\mathcal{F}_{\beta}(M)$, and any edge intersecting $K'$ is either contained in $K'$, or intersects $P \setminus v$ in one vertex and $C_2$ in $r-1$ vertices, or intersects $Q \setminus w$ in one vertex and $C_1$ in $r-1$ vertices.
\item If $P$ is valid w.r.t. $\mathcal{F}_\alpha$ then for any valid partition $(C_1, C_2)$ of $C(P, Q,v,w)$ w.r.t. $\mathcal{F}_\alpha(M)$ and $\mathcal{F}_\beta(M)$ we have $P \subseteq C_1$.

If $P$ is minimal invalid w.r.t. $\mathcal{F}_\alpha$ then $C(P, Q, v, w)$ has no valid partition w.r.t. $\mathcal{F}_{\alpha}(M)$ and $\mathcal{F}_{\beta}(M)$.
\end{enumerate}
\end{lemma}

Before we proceed to the proof of the Rigidity Lemma we note the following consequence, which we will also use in the proof.

\begin{lemma}[Addendum to the Rigidity Lemma]
Under the hypotheses of Lemma \ref{lem:rigid}, the following holds.
\renewcommand{\theenumi}{(\alph{enumi})}
\begin{enumerate}
\item If $P$ is valid w.r.t. $\mathcal{F}_\alpha$ then $C(P, Q, v, w)$ has a valid partition $(C_1, C_2)$ w.r.t. $\mathcal{F}_\alpha(M)$ and $\mathcal{F}_\beta(M)$, with $C_1$ containing an induced copy $P'$ of $P$. Furthermore any edge intersecting $P'$ is either contained in $P'$ or intersects $P'$ in one vertex and $C_2$ in $r-1$ vertices.
\item If $P$ is minimal invalid w.r.t. $\mathcal{F}_\alpha$ then $C(P, Q, v, w)$ has a partition $(C_1, C_2)$ with $C_1$ containing an induced copy $P'$ of $P$, such that $(C_1 \setminus P', C_2)$ is a valid partition of $C(P, Q, v, w) \setminus P'$ w.r.t. $\mathcal{F}_\alpha(M)$ and $\mathcal{F}_\beta(M)$.  Furthermore any edge intersecting $P'$ is either contained in $P'$ or intersects $P'$ in one vertex and $C_2$ in $r-1$ vertices.
\end{enumerate}
\end{lemma}
\begin{proof}
By (B) of the Rigidity Lemma, $C(P, Q, v, w)$ contains an induced copy $K'$ of $K$ such that $C(P,Q,v,w) \setminus K'$ has a valid partition $(C_1', C_2')$ w.r.t. $\mathcal{F}_\alpha(M)$ and $\mathcal{F}_\beta(M)$. Let $P_1$ and $Q_1$ be the copies of $P \setminus v$, respectively $Q \setminus w$, in $K'$, and $z$ the unique vertex in $K' \setminus (P_1 \cup Q_1)$.

Define $C_1 := C_1' \dot\cup (P_1 \cup \{z\})$ and $C_2 := C_2' \dot\cup Q_1$. Note that $C_1 \simeq C_1' \dot\cup P$, as $P_1$ and $z$ form an induced copy $P'$ of $P$ in $C_1$.

By (B), any edge intersecting $P'$ is either contained in $P'$ or intersects $P'$ in one vertex and $C_2$ in $r-1$ vertices.

As $Q_1$ is valid w.r.t. $\mathcal{F}_\beta$, $C_2$ is always valid w.r.t. $\mathcal{F}_\beta(M)$ by Lemma \ref{lem:valid}, (iii) and (iv).

Moreover if $P$ is valid w.r.t. $\mathcal{F}_\alpha$ then $C_1$ is also valid w.r.t. $\mathcal{F}_\alpha(M)$ by Lemma \ref{lem:valid}, (iii) and (iv). This proves (a).

Finally, if $P$ is minimal invalid w.r.t. $\mathcal{F}_\alpha$ then $C_1 \setminus P' = C_1'$ is valid w.r.t. $\mathcal{F}_\alpha(M)$. This proves (b).
\end{proof}

\begin{proof}[Proof of the Rigidity Lemma]
Define $M_1 := v(Q)$ and by induction on $k \geq 2$ define the positive integer
\begin{equation}
\label{eq:MF}
M_k := M_{k-1} + 2^{M_{k-1}}(2v(Q)-2+k).
\end{equation}

We shall show that the lemma holds for $M_{P, Q, v, w} := M_{v(P)}$. Let $M \geq M_{P, Q, v, w}$ be arbitrary and consider any $M$-closures $\mathcal{F}_\alpha(M)$ and $\mathcal{F}_\beta(M)$.

We prove by induction on $v(P) \geq 1$ that an $r$-graph $C(P, Q, v, w)$ with the desired properties exists.

First assume that $v(P) = 1$. Then $P$ is just a vertex $v$ and necessarily $P$ is valid w.r.t. $\mathcal{F}_\alpha$.

Define $C(P, Q, v, w) := Q$. Then $v(C(P, Q, v, w)) = M_1 = M_{P, Q, v, w}$, proving (A).

Clearly $Q$ is isomorphic to the $r$-graph $K'$ prescribed by (B), and the empty graph has always a valid partition $(\emptyset, \emptyset)$ w.r.t. $\mathcal{F}_\alpha(M)$ and $\mathcal{F}_\beta(M)$. This proves (B).

Finally, let $(C_1, C_2)$ be any valid partition of $C(P, Q, v, w)$ w.r.t. $\mathcal{F}_\alpha(M)$ and $\mathcal{F}_\beta(M)$. If $C_1 = \emptyset$ then $C_2 = Q$, contradicting the validity of $C_2$ w.r.t. $\mathcal{F}_\beta(M)$ by Lemma \ref{lem:valid}, (v). Therefore $C_1 \neq \emptyset$, and hence $P \subseteq C_1$, proving (C).

Now assume that $v(P) > 1$ and the induction hypothesis holds for all $r$-graphs $P'$ on fewer vertices, such that $P'$ is either valid w.r.t. $\mathcal{F}_\alpha$, or belongs to $\mathcal{F}_\alpha$ and is minimal invalid with respect to it.

If $P$ is valid, by Lemma \ref{lem:valid}, (i), so is $P\setminus v$. If $P$ is minimal invalid, then $P \setminus v$ is valid by definition. Thus in any case $P \setminus v$ is valid. Fix $v' \in V(P \setminus v)$ arbitrary. As $M \geq M_{P\setminus v, Q, v', w}$, the induction hypothesis gives us an $r$-graph $C(P \setminus v) := C(P\setminus v, Q, v', w)$ satisfying (A)-(C).

We define a sequence of $r$-graphs $F_0, F_1, \ldots$ as follows.

Let $(C_1^0, C_2^0)$ be the valid partition of $C(P \setminus v)$ w.r.t. $\mathcal{F}_\alpha(M)$ and $\mathcal{F}_\beta(M)$, guaranteed by the Addendum to the Rigidity Lemma, (a). Let $Q_1^0$ and $Q_2^0$ be vertex disjoint copies of $Q \setminus w$, and $P_1^0$ a copy of $P \setminus v$. To the $r$-graph $C_1^0 \times (C_2^0 \dot\cup Q_1^0)$ add $P_1^0$ and all edges intersecting $P_1^0$ in one vertex and $Q_1^0$ in $r-1$ vertices. Then add $Q_2^0$ and all edges intersecting $Q_2^0$ in one vertex and $C_1^0$ in $r-1$ vertices. Finally, add a vertex $z_0$ and edges in such a way that $P_1^0 \cup \{z_0\}$ induces a copy of $P$, and $Q_2^0 \cup \{z_0\}$ induces a copy of $Q$. No other edges incident with $z_0$ are added. This defines $F_0$.

Now suppose $i \geq 0$ and we have constructed $F_i$.

First assume there exists a partition $(D_1^{i+1}, D_2^{i+1})$ of $F_i$ valid w.r.t. $\mathcal{F}_\alpha(M)$ and $\mathcal{F}_\beta(M)$ such that $P \not\subseteq D_1^{i+1}$. We construct $F_{i+1}$ in a similar manner as above. Let $Q_1^{i+1}$ and $Q_2^{i+1}$ be vertex disjoint copies of $Q \setminus w$, and let $P_1^{i+1}$ be a copy of $P \setminus v$. To the $r$-graph $D_1^{i+1} \times (D_2^{i+1} \dot\cup Q_1^{i+1})$ add $P_1^{i+1}$ and all edges intersecting $P_1^{i+1}$ in one vertex and $Q_1^{i+1}$ in $r-1$ vertices. Then add $Q_2^{i+1}$ and all edges intersecting $Q_2^{i+1}$ in one vertex and $D_1^{i+1}$ in $r-1$ vertices. Finally, add a vertex $z_{i+1}$ and edges in such a way that $P_1^{i+1} \cup \{z_{i+1}\}$ induces a copy of $P$, and $Q_2^{i+1} \cup \{z_{i+1}\}$ induces a copy of $Q$. No other edges incident with $z_{i+1}$ are added. This defines $F_{i+1}$.

If no partition $(D_1^{i+1}, D_2^{i+1})$ with the desired properties exists, we set $C(P, Q, v, w) := F_i$ and we stop.

\begin{claim}
\label{clm:finite}
The sequence has at most $2^{v(C(P \setminus v))}$ terms.
\end{claim}
\begin{proof}
Note that
\begin{equation*}
C(P \setminus v) \subseteq F_0 \subseteq F_1 \subseteq F_2 \ldots.
\end{equation*}
We identify the vertices of these graphs in such a way that all the above inclusion maps are given by the identity. Therefore we can speak about "the" subgraph $C(P \setminus v)$ of $F_i$,
although $F_i$ may possibly contain other copies of $C(P \setminus v)$.

For each $i \geq 1$, set $C_1^i := D_1^{i} \cap V(C(P \setminus v))$ and $C_2^i := D_2^i \cap V(C(P \setminus v))$. Then by Lemma \ref{lem:valid}, (i), $(C_1^i, C_2^i)$ is a valid partition of $C(P \setminus v)$ w.r.t. $\mathcal{F}_\alpha(M)$ and $\mathcal{F}_\beta(M)$. Note that $(C_1^0, C_2^0)$ was already defined.

We claim that all the partitions $(C_1^i, C_2^i)$ must be distinct.

Indeed, suppose for a contradiction that there are $0 \leq i < j$ such that $(C_1^i, C_2^i) = (C_1^j, C_2^j)$.

By assumption, $D_1^j \supseteq C_1^j = C_1^i$ contains no copy of $P$. As $(C_1^i, C_2^i)$ is valid, it follows by (C) that $C_1^i$ and hence $D_1^j$ contains a copy $P'$ of $P \setminus v$. But $F_i \subseteq F_{j-1} = D_1^j \cup D_2^j$ contains the copy $Q_1^i$ of $Q \setminus w$. As any vertex $x$ of $Q_1^i$ together with $P'$ would induce the subgraph $P' \times x \supseteq P$, it follows that all vertices of $Q_1^i$ are part of $D_2^j$. A similar argument shows that $V(Q_2^i) \subseteq V(D_2^j)$. But $F_i$ also contains the copy $P_1^i$ of $P \setminus v$, and $P_1^i \cup \{z_i\}$ forms a copy of $P$. As $P \not\subseteq D_1^j$, for some $x \in V(P_1^i) \cup \{z_i\}$ we have $x \in D_2^j$.

If $x \neq z_i$, then $Q \subseteq x \times Q_1^i \subseteq D_2^j$. 

If $x = z_i$ then $x$ and $Q_2^i$ induce a copy of $Q$ in $D_2^j$.

Thus in any case $Q \subseteq D_2^j$. But this is a contradiction with the validity of the partition $(D_1^j, D_2^j)$.

Consequently all partitions $(C_1^i, C_2^i)$ of $C(P \setminus v)$ must be distinct. There are at most $2^{v(C(P \setminus v))}$ such partitions, completing the proof.
\end{proof}

Let $s$ be the length of the sequence. By Claim \ref{clm:finite}, $1 \leq s \leq 2^{v(C(P \setminus v))}\leq 2^{M_{v(P) - 1}}$.

By induction on $0 \leq i \leq s-1$, we see that $v(F_i) \leq v(C(P \setminus v))+(i+1)(2v(Q) - 2 + v(P))$. Therefore $v(C(P, Q, v, w)) \leq M_{v(P) - 1} + 2^{M_{v(P) - 1}}(2v(Q) - 2+v(P)) = M_{v(P)} = M_{P, Q, v, w}$. This proves (A).

Define $D_1^0 := C_1^0$ and $D_2^0 := C_2^0$. Note that $C(P, Q, v, w) = F_{s-1}$ contains a copy $K'$ of $K$ given by $P_1^{s-1}$, $Q_2^{s-1}$ and $z_{s-1}$, and $C(P, Q, v, w) \setminus K'$ has the partition $(D_1^{s-1}, D_2^{s-1} \dot\cup Q_1^{s-1})$.

It is easy to see that this partition is valid w.r.t. $\mathcal{F}_\alpha(M)$ and $\mathcal{F}_\beta(M)$, proving (B).

As the sequence stopped, any valid partition $(D_1, D_2)$ of $C(P, Q, v, w)$ w.r.t. $\mathcal{F}_\alpha(M)$ and $\mathcal{F}_\beta(M)$ must satisfy $P \subseteq D_1$. If $P$ is minimal invalid, this is a contradiction with Lemma \ref{lem:valid}, (ii) and (v), and therefore in this case no valid partition $(D_1, D_2)$ can exist. This proves (C), finishing the proof of the lemma.
\end{proof}

\subsection{\normalsize The Collapsing Lemma}

We continue to follow the strategy of Pikhurko from \cite{pikhurko12}. The next step is to prove a \textit{collapsing lemma}. The general idea is to show that under certain circumstances two $r$-graphs which are extremal and "close" to one another must in fact be isomorphic (thus the extremal structure "collapses" onto a predefined pattern). This technique goes back to the work of Simonovits in the 60s \cite{simonovits68} and was later developed as a tool for the exact determination of Tur\'an densities.

To the best of our knowledge, all the previous applications of this method considered that one of the two $r$-graphs is a blow-up or an iterated blow-up structure. In such a situation many nice properties are available, most importantly, the addition of any new edge to this graph creates $\Omega(n^{v(F) - r})$ copies of some forbidden graph $F$. Thus even a small, local modification requires the deletion of many edges to maintain the property of being $F$-free, in particular the resulting graph can not be "close" to the initial one, and hence if it is "too close", it must be equal.

However, this is too much to expect in our present situation; for an arbitrary Tur\'an density there is no structure to use, and there is no evidence that adding an edge to an extremal $r$-graph would create many copies of some forbidden subgraph.

Nevertheless in some cases local modifications create many forbidden $r$-graphs; this can be read out of the $\pi$ function, and it was our goal in the previous section to extract this information. We shall use it to show that the number of edges in an extremal $r$-graph is suitably bounded by the function $\oplus_r$.

First we need a couple of definitions.

Let $G$ and $H$ be two $r$-graphs with the same number $n$ of vertices. For $\eps > 0$, we say that $G$ and $H$ are $\eps$-close if $G$ is isomorphic to an $r$-graph $G'$ on $V(H)$ such that $|E(G') \Delta E(H)| \leq \eps\binom{n}{r}$. In other words, we can obtain $H$ from $G$ by adding or deleting at most $\eps\binom{n}{r}$ edges.

A family of $r$-graphs $\mathcal{F}$ is called \textit{minimal} if it is weakly closed under homomorphisms and any $r$-graph $F \in \mathcal{F}$ is minimal invalid w.r.t. $\mathcal{F}$.

\begin{lemma}
\label{lem:minimal}
For any finite family of $r$-graphs $\mathcal{F}$ there exists a finite minimal family of $r$-graphs $\mathcal{F}'$ with $\pi(\mathcal{F}') = \pi(\mathcal{F})$.
\end{lemma}
\begin{proof}
Clearly we may assume that each $r$-graph in $\mathcal{F}$ has at least one edge.

We shall repeatedly apply one of the following operations to $\mathcal{F}$.

(O1) If $F \in \mathcal{F}, f : F \rightarrow F'$ is a surjective homomorphism and $F' \notin \mathcal{F}$, then add $F'$ to $\mathcal{F}$.

(O2) If $F \in \mathcal{F}$, $F' \notin \mathcal{F}$ is a proper subgraph of $F$ with at least one edge and $\pi(\mathcal{F} \cup \{F'\}) = \pi(\mathcal{F})$, then add $F'$ to $\mathcal{F}$.

(O3) If $F' \subsetneq F$ and $F', F \in \mathcal{F}$ then remove $F$ from $\mathcal{F}$.

We start by applying (O1) and (O2) until none of these operations can be applied anymore, and then we apply (O3) as many times as possible. As (O1), (O2) and (O3) can each be applied only a finite number of times, we obtain a finite family $\mathcal{F}'$ of $r$-graphs.

We first claim that $\pi(\mathcal{F}') = \pi(\mathcal{F})$ and that any element of $\mathcal{F}'$ has at least one edge. To prove this we examine each operation separatedly.

Consider (O1). Let $F \in \mathcal{F}$ and suppose that there exists $F'$ and a surjective homomorphism from $F$ to $F'$ such that $F' \notin \mathcal{F}$. Then for some $t \geq 1$, $F \subseteq F'(t)$. Hence by Lemma \ref{lem:piprop}, (i), (ii) and (iv),
\begin{equation*}
\pi(\mathcal{F}) \geq \pi(\mathcal{F} \cup \{F'\}) = \pi(\mathcal{F} \cup \{F'(t)\}) \geq \pi(\mathcal{F} \cup \{F\}) = \pi(\mathcal{F}),
\end{equation*}
proving that $\pi(\mathcal{F} \cup \{F'\}) = \pi(\mathcal{F})$. Thus we can add $F'$ to $\mathcal{F}$ without changing the Tur\'an density. Moreover $F'$ has at least one edge, as it contains a homomorphic image of $F$.

Now consider (O2). Let $F \in \mathcal{F}$ and $F' \notin \mathcal{F}$ be a proper subgraph of $F$ with at least one edge, such that $\pi(\mathcal{F} \cup \{F'\}) = \pi(\mathcal{F})$. Then we can add $F'$ to $\mathcal{F}$, and this does not change the Tur\'an density.

Finally, by Lemma \ref{lem:piprop}, (i) and (ii), whenever $F' \subsetneq F$ and $F', F \in \mathcal{F}$, we can remove $F$ from $\mathcal{F}$ without changing the Tur\'an density.

Consequently it follows by induction that $\pi(\mathcal{F}') = \pi(\mathcal{F})$ and any element of $\mathcal{F}'$ has at least one edge. In particular, any $F \in \mathcal{F}'$ is invalid w.r.t. $\mathcal{F}'$.

We now claim that any $r$-graph in $\mathcal{F}'$ is also minimal invalid with respect to it. Indeed, suppose for a contradiction that there exists $F \in \mathcal{F}'$ and $x \in V(F)$ such that $F \setminus x$ is invalid w.r.t. $\mathcal{F}'$. Then $F \setminus x$ has at least one edge and $\pi(\mathcal{F}' \cup \{F \setminus x\}) = \pi(\mathcal{F}') = \pi(\mathcal{F})$. By (O2), $F \setminus x \in \mathcal{F}$ before the first application of (O3). But then (O3) forces the removal of $F$ from $\mathcal{F}$, a contradiction.

Consequently any $r$-graph $F \in \mathcal{F}'$ is minimal invalid w.r.t. $\mathcal{F}'$.

Finally, we claim $\mathcal{F}'$ is weakly closed under homomorphisms. Indeed, let $F \in \mathcal{F}'$ and assume $f : F \rightarrow F'$ is a surjective homomorphism. Then $F'$ was added to $\mathcal{F}$ by (O1) if it was not already present in $\mathcal{F}$. If $F' \notin \mathcal{F}'$, then it must have been removed by (O3). Consequently there exists $F'' \subsetneq F'$ such that $F'' \in \mathcal{F}'$. This shows that $\mathcal{F}'$ is weakly closed under homomorphisms.
\end{proof}

Note that if $\mathcal{F}$ is minimal and $\pi(\mathcal{F}) = 0$, then $\mathcal{F} = \{e\}$, where $e$ is the $r$-edge.

\begin{lemma}[The Collapsing Lemma]
\label{lem:collapse}
Let $\mathcal{F}_\alpha$ and $\mathcal{F}_\beta$ be two non-empty finite minimal families of $r$-graphs with $\pi(\mathcal{F}_\alpha) = \alpha$ and $\pi(\mathcal{F}_\beta) = \beta$.

Then there exists an $M_{\alpha, \beta} > 0$ such that for any $M \geq M_{\alpha, \beta}$ and any $M$-closures $\mathcal{F}_{\alpha}(M)$ and $\mathcal{F}_{\beta}(M)$ the following holds.

There exist a finite family of $r$-graphs $\mathcal{F}_{\alpha, \beta}$ and an $\eps > 0$ with the following properties.
\begin{itemize}
\item If $H_1$ is an $\mathcal{F}_\alpha(M)$-free $r$-graph and $H_2$ is an $\mathcal{F}_\beta(M)$-free $r$-graph then $H_1 \times H_2$ is $\mathcal{F}_{\alpha, \beta}$-free.
\item Furthermore for any $\zeta > 0$ there exists an $n_0 \geq 1$ such that any maximum $\mathcal{F}_{\alpha, \beta}$-free $r$-graph $G$ on $n \geq n_0$ vertices which is $\eps$-close to an $r$-graph of the form $H_1 \times H_2$, with $H_1$ an $\mathcal{F}_\alpha(M)$-free $r$-graph, and $H_2$ an $\mathcal{F}_\beta(M)$-free $r$-graph, has at most
\begin{equation}
\label{eq:n_edges}
(\alpha \oplus_r \beta + \zeta)\binom{n}{r}
\end{equation}
edges.
\end{itemize}
\end{lemma}
\begin{proof}
For any pair $(F_\alpha, F_\beta) \in \mathcal{F}_\alpha \times \mathcal{F}_\beta$, and any choice of $v \in V(F_\alpha), w \in V(F_\beta)$, the Rigidity Lemma gives us a positive integer $M_{F_\alpha, F_\beta, v, w}$. We set $M_{\alpha, \beta}$ to be the maximum of these values.

Now let $M \geq M_{\alpha, \beta}$ be arbitrary and fix arbitrary $M$-closures $\mathcal{F}_{\alpha}(M)$ and $\mathcal{F}_{\beta}(M)$.

For any pair $(F_\alpha, F_\beta) \in \mathcal{F}_\alpha \times \mathcal{F}_\beta$, and any choice of $v \in V(F_\alpha), w \in V(F_\beta)$, the Rigidity Lemma now gives us an $r$-graph $C(F_\alpha, F_\beta, v, w)$. We set
\begin{equation*}
\mathcal{F}_{\alpha, \beta}^{*} := \{C(F_\alpha, F_\beta, v, w) : F_\alpha \in \mathcal{F}_\alpha, F_\beta \in \mathcal{F}_\beta, v \in V(F_\alpha), w \in V(F_\beta)\}
\end{equation*}
and let $\mathcal{F}_{\alpha, \beta}$ be the closure of $\mathcal{F}_{\alpha, \beta}^{*}$ under homomorphisms.

For each $F_\alpha \in \mathcal{F}_\alpha$, choose an arbitrary $F_\beta \in \mathcal{F}_\beta$ along with an arbitrary $v \in  V(F_\alpha)$ and $w \in V(F_\beta)$. With these choices, define the $r$-graph $C(F_\alpha) := C(F_\alpha, F_\beta, v, w)$. The exact choices we make are irrelevant for the argument to follow. Similarly for any $F_\beta \in \mathcal{F}_\beta$, choose $F_\alpha \in \mathcal{F}_\alpha, v \in  V(F_\alpha)$ and $w \in V(F_\beta)$ arbitrary, and define the $r$-graph $C(F_\beta) := C(F_\alpha, F_\beta, v, w)$.

Before we go any further we establish the first part of the lemma.

\begin{claim}
\label{clm:free}
If $H_1$ is an $\mathcal{F}_\alpha(M)$-free $r$-graph and $H_2$ is an $\mathcal{F}_\beta(M)$-free $r$-graph then $H := H_1 \times H_2$ is $\mathcal{F}_{\alpha, \beta}$-free.
\end{claim}
\begin{proof}
Let $P \in \mathcal{F}_\alpha, Q \in \mathcal{F}_\beta$ and $v \in V(P), w \in V(Q)$. Let $f : C(P, Q, v, w) \rightarrow F$ be any surjective homomorphism. We show $H$ is $F$-free.

Suppose for a contradiction that $F$ embeds into $H = H_1 \times H_2$. Then there exists a partition $(T_1, T_2)$ of $C(P, Q, v, w)$ such that $f(T_1) \subseteq H_1$ and $f(T_2) \subseteq H_2$. By Lemma \ref{lem:rigid}, (C), this partition is not valid. Thus w.l.o.g. we may assume that $T_1$ is not valid w.r.t. $\mathcal{F}_\alpha(M)$.

Let $t \geq 1$ such that $T_1 \subseteq f(T_1)(t)$.

If $\alpha > 0$ then by Lemma \ref{lem:piprop}, (iv) and (ii),
\begin{equation*}
\pi(\mathcal{F}_\alpha(M) \cup \{f(T_1)\}) = \pi(\mathcal{F}_\alpha(M) \cup \{f(T_1)(t)\}) \geq \pi(\mathcal{F}_\alpha(M) \cup \{T_1\}) = \pi(\mathcal{F}_\alpha(M)).
\end{equation*}
Consequently $\pi(\mathcal{F}_\alpha(M) \cup \{f(T_1)\}) = \pi(\mathcal{F}_\alpha(M))$. However, $v(f(T_1)) \leq v(T_1) \leq v(C(P, Q, v, w)) \leq M_{\alpha, \beta} \leq M$ by definition. As $\alpha > 0$, $T_1$ and hence $f(T_1)$ certainly contains at least one edge. Thus by maximality, $f(T_1) \in \mathcal{F}_{\alpha}(M)$. This contradicts the fact that $H_1$ is $\mathcal{F}_\alpha(M)$-free.

Hence $\alpha = 0$. But then $\lambda(H_1) \geq \lambda(f(T_1)) \geq \lambda(T_1) > 0$ by Lemma \ref{lem:lambdaprop}, (ii). Consequently $H_1$ has at least one edge. But $\mathcal{F}_\alpha$ contains the one edge $r$-graph, again contradicting the fact that $H_1$ is $\mathcal{F}_\alpha$-free. This proves the claim.
\end{proof}

We now prove the second part of the Collapsing Lemma. Define
\begin{equation*}
\eps_0 := \frac{1}{2}\min\{\alpha - \theta(\mathcal{F}_\alpha(M)), \beta  - \theta(\mathcal{F}_\beta(M))\},
\end{equation*}
where recall that $\theta(\mathcal{F}(M))$ is the threshold of the $M$-closure $\mathcal{F}(M)$. By definition, $\theta(\mathcal{F}_{\alpha}(M)) = -1$ if $\alpha = 0$, and similarly $\theta(\mathcal{F}_{\beta}(M)) = -1$ if $\beta = 0$. Hence we always have $\eps_0 > 0$.

Then define
\begin{align*}
\delta &:= \min\{\delta(\eps_0, \mathcal{F}_\alpha(M)), \delta(\eps_0, \mathcal{F}_\beta(M))\},\\
n^* &:= \max\{n^*(\eps_0, \mathcal{F}_\alpha(M)), n^*(\eps_0, \mathcal{F}_\beta(M))\}.
\end{align*}

We are now going to define several other constants. Rather than giving a precise definition, we shall list a (admitedly long) list of inequalities they have to satisfy, and it will be obvious from this list that a choice satisfying all the given inequalities can be made. The reader may safely skip this part of the proof and return later when needed.

Note that $\alpha < 1$ and $\beta < 1$ by Lemma \ref{lem:piprop}, (v), as $\mathcal{F}_\alpha$ and $\mathcal{F}_\beta$ are both non-empty.

Moreover, recall that in Lemma \ref{lem:max2} we defined
\begin{equation*}
\mathfrak{g}_{\alpha, \beta}(x) = \alpha x^r + \beta (1-x)^r + r!\sum_{j=1}^{r-1}\frac{x^j(1-x)^{r-j}}{j!(r-j)!}
\end{equation*}
and $x_{\alpha, \beta} = \frac{\sqrt[r-1]{1 - \beta}}{\sqrt[r-1]{1-\alpha} + \sqrt[r-1]{1-\beta}}$.

Recall also the definition of $\tau$ in Lemma \ref{lem:add_vertex}.

Now choose constants
\begin{equation*}
0 < \eps \ll c_1 \ll c_2 \ll c_3 \ll c_4 \ll c_5 \ll c_6
\end{equation*}
such that the following conditions hold:
\begin{align}
\eps &< \min\left\{\frac{1-\beta}{3}\left(1 - \frac{1}{(1+\sqrt[r-1]{1-\beta})^{r-1}}\right),
\frac{1-\alpha}{3}\left(1 - \frac{1}{(1+\sqrt[r-1]{1-\alpha})^{r-1}}\right)\right\},\label{eq:eps1}\\%
%%%%%%
c_1 &< \min\{x_{\alpha, \beta}, 1 - x_{\alpha, \beta}\},\\%
%%%%%%
\eps &< \frac{1}{4}\left(\alpha \oplus_r \beta - \max\{\mathfrak{g}_{\alpha, \beta}(x_{\alpha, \beta} - c_1), \mathfrak{g}_{\alpha, \beta}(x_{\alpha, \beta} + c_1)\}\right),\label{eq:eps2}\\%
%%%%%%
\eps &< \frac{c_2}{3}\min\left\{\left(\frac{\sqrt[r-1]{1-\alpha}}{1+\sqrt[r-1]{1-\alpha}}\right)^r, 
\left(\frac{\sqrt[r-1]{1-\beta}}{1+\sqrt[r-1]{1-\beta}}\right)^r\right\},\label{eq:eps3}\\%
%%%%%%
\eps^{1/2} &< c_3r \min\{(1-x_{\alpha,\beta}-2c_1)^{r-1}(x_{\alpha,\beta}-c_1), (x_{\alpha,\beta}-2c_1)^{r-1}(1-x_{\alpha,\beta}-c_1)\},\label{eq:eps4}\\%
%%%%%%
\eps &< r!\frac{\delta^2}{8(r+1)}(1-c_3)^{M_{\alpha, \beta}}(x_{\alpha,\beta} - c_1)^r(1-x_{\alpha,\beta}-c_1)^r,\label{eq:eps5}\\%
%%%%%%
c_6 &< (r-1)!\frac{\delta^2}{16M_{\alpha, \beta}}(1-c_3)^{M_{\alpha, \beta}},\label{eq:eps6}\\
%%%%%%
\eps^{1/2} + c_3r &< c_6,\label{eq:eps7}\\%
%%%%%%
c_5 &> c_3r + c_4 + (4^r+2^{r-1}+2)c_1,\label{eq:c5}\\%
%%%%%%
c_6 &> \frac{(x_{\alpha,\beta}+c_1)^{r-1} - x_{\alpha, \beta}^{r-1} + c_5}{(x_{\alpha, \beta}+c_1)^{r-1}},\label{eq:c6}\\%
%%%%%%
c_2+c_3r &< \min\{\alpha - \theta(\mathcal{F}_{\alpha}(M)) - \eps_0, \beta - \theta(\mathcal{F}_\beta(M)) - \eps_0\},\label{eq:c231}\\%
%%%%%%
c_2+c_3r &< \min\{\tau(\mathcal{F}_\alpha, c_4), \tau(\mathcal{F}_\beta, c_4)\}\label{eq:c232}.
\end{align}

In order for this system of inequalities to have a solution, it is enough if the following condition holds. For each inequality, the smaller quantity tends to zero when all the unknowns appearing in it tend to zero, and any unknown $c_i$ appearing in the greater quantity has index strictly larger than any unknown $c_j$ appearing in the smaller quantity. The only problematic inequality is \eqref{eq:eps6}; however after a rearrangement it can be seen that it also satisfies this condition.

Furthermore recall that if $\alpha = 0$ then $\theta(\mathcal{F}_\alpha(M)) = -1$, and hence if $\alpha = \beta = 0$ then $\eps_0 = \frac{1}{2}$, and so the right-hand side of \eqref{eq:c231} is in this case equal to $\frac{1}{2}$.

Now let $\zeta > 0$ be arbitrary. We also choose constants 
\begin{equation*}
n^{*} \ll n_1  \ll n_2 \ll n_0
\end{equation*}
in the following way.

Recall the definition of $n_C$ in Lemma \ref{lem:remset}. We require that
\begin{equation}
\label{eq:n_1C3}
n_1 \geq n_C(\delta, \max\{v(F_\alpha), v(F_\beta)\}, v(C(F_\alpha, F_\beta, v, w))),
\end{equation}
for any choice of $F_\alpha \in \mathcal{F}_\alpha, F_\beta \in \mathcal{F}_\beta, v \in V(F_\alpha)$ and $w \in V(F_\beta)$.

Recall the definition of $n_V$ in Lemma \ref{lem:add_vertex}. We require
\begin{equation}
\label{eq:n_1V}
n_1 \geq \max\left\{n_V(\mathcal{F}_\alpha, c_4), n_V(\mathcal{F}_\beta, c_4)\right\},
\end{equation}
and that for any $n \geq n_1$, $\ex(n, \mathcal{F}_\alpha) < (\alpha + \frac{\zeta}{4})\binom{n}{r}$ and $\ex(n, \mathcal{F}_\beta) < (\beta + \frac{\zeta}{4})\binom{n}{r}$.

Once $n_1$ is fixed, we choose $n_2$ such that
\begin{equation}
\label{eq:n_21}
n_2 \geq \frac{n_1}{1 - c_3},
\end{equation}
and such that for any $n \geq n_2$, $\ex(n, \mathcal{F}_\alpha) < (\alpha + \eps)\binom{n}{r}$ and $\ex(n, \mathcal{F}_\beta) < (\beta + \eps)\binom{n}{r}$.

Finally, we choose $n_0$ such that
\begin{equation}
\label{eq:n_01}
n_0 \geq \max\left\{\frac{n_2 r}{\eps}, \frac{r}{c_1}, n_D(\mathcal{F}_{\alpha, \beta}, c_1)\right\},
\end{equation}
where $n_D$ is given by Lemma \ref{lem:min_degree} (the assumptions of the lemma are satisfied because $\mathcal{F}_{\alpha, \beta}$ is closed under homomorphisms), and such that
\begin{equation}
\label{eq:n_03}
\prod_{i=1}^{r-1}\left(1 - \frac{i}{n_0}\right) \geq \max\left\{ \frac{\alpha \oplus_r \beta - 2c_1}{\alpha \oplus_r \beta - c_1}, \frac{\alpha \oplus_r \beta - \eps}{\alpha \oplus_r \beta}, \frac{\alpha \oplus_r \beta + \frac{\zeta}{2}}{\alpha \oplus_r \beta + \zeta}\right\}.
\end{equation}
Note the dependency of $n_0$ and $n_1$ on $\zeta$.

For the rest of the proof we shall assume that $\pi(\mathcal{F}_{\alpha, \beta}) \geq \alpha \oplus_r \beta$, otherwise we can choose $n_0$ large enough so that the lemma trivially holds.

Finally, let $G$ be any maximum $\mathcal{F}_{\alpha, \beta}$-free $r$-graph $G$ on $n \geq n_0$ vertices and suppose $G$ is $\eps$-close to an $r$-graph of the form $H := H_1 \times H_2$, with $H_1$ an $\mathcal{F}_\alpha(M)$-free $r$-graph, and $H_2$ an $\mathcal{F}_\beta(M)$-free $r$-graph. We shall assume for a contradiction that $G$ has more than $(\alpha \oplus_r \beta + \zeta)\binom{n}{r}$ edges.

We identify the vertex set of $G$ with that of $H$ in such a way that $|E(G) \Delta E(H)| \leq \eps\binom{n}{r}$. Let $(G_1, G_2)$ be the partition of $G$ so that $V(G_1) = V(H_1)$ and $V(G_2) = V(H_2)$.

The proof now begins in earnest. We start with the following claim.

\begin{claim}
\label{clm:clm1}
$G_1$ and $G_2$ have each at least $n_2$ vertices.
\end{claim}
\begin{proof}
We only prove that $v(G_1) \geq n_2$, as the other statement is proved similarly.

Assume for a contradiction that this is not the case. By \eqref{eq:n_01}, $n_0 \geq 2n_2$, in particular $v(G_2) \geq n_2$.

Hence
\begin{align*}
e(G) &\leq e(G_2) + n_2\binom{n-1}{r-1}\\
&\leq e(H_2) + n_2\binom{n-1}{r-1} + \eps\binom{n}{r}\\
&\leq (\beta + 3\eps)\binom{n}{r}, \textrm{ by our choice of $n_2$ and \eqref{eq:n_01}},\\
&< (\alpha \oplus_r \beta)\binom{n}{r},
\end{align*}
where the last inequality follows from \eqref{eq:eps1} and the inequality $\alpha \oplus_r \beta \geq 0 \oplus_r \beta = 1 - \frac{1-\beta}{(1+\sqrt[r-1]{1-\beta})^{r-1}}$.

This is a contradiction, proving the claim.
\end{proof}

Now set $a := v(G_1)$ and $b := v(G_2)$. Note that $a+b = n$.

\begin{claim}
\label{clm:clm2}
$\frac{a}{n} \in (x_{\alpha, \beta} - c_1, x_{\alpha, \beta} + c_1)$ and $\frac{b}{n} \in (1-x_{\alpha,\beta} - c_1, 1 - x_{\alpha, \beta} + c_1)$.
\end{claim}
\begin{proof}
We only prove that $\frac{a}{n} \in (x_{\alpha, \beta} - c_1, x_{\alpha, \beta} + c_1)$, as the other statement follows from this one.

Suppose for a contradiction that this is not the case.

Note that
\begin{align*}
e(G) &\leq e(H_1) + e(H_2) + \sum_{i=1}^{r-1} \binom{a}{i}\binom{b}{r-i} + \eps\binom{n}{r}\\
&\leq (\alpha + \eps)\binom{a}{r} + (\beta+\eps)\binom{b}{r} + \sum_{i=1}^{r-1}\binom{a}{i}\binom{b}{r-i} + \eps\binom{n}{r}, \textrm{by Claim \ref{clm:clm1} and our choice of $n_2$},\\
&\leq \alpha\frac{a^r}{r!} + \beta\frac{b^r}{r!} + \sum_{i=1}^{r-1} \frac{a^ib^{r-i}}{i!(r-i)!}+3\eps \frac{n^r}{r!}\\
&=(\mathfrak{g}_{\alpha, \beta}(\frac{a}{n}) + 3\eps)\frac{n^r}{r!}.
\end{align*}
As $\mathfrak{g}_{\alpha, \beta}$ is strictly concave, it follows by \eqref{eq:eps2} that
\begin{equation*}
e(G) \leq (\alpha \oplus_r \beta - \eps)\frac{n^r}{r!} \leByRef{eq:n_03} (\alpha \oplus_r \beta)\binom{n}{r},
\end{equation*}
a contradiction.
\end{proof}

\begin{claim}
\label{clm:clm3}
$G_1$ has density at least $\alpha - c_2$ and $G_2$ has density at least $\beta - c_2$.
\end{claim}
\begin{proof}
We only prove that $d(G_1) \geq \alpha - c_2$, as the other statement is proved similarly.

Assume for a contradiction that this is not the case.

As in the previous claim, note that
\begin{align*}
e(G) &\leq d(G_1)\frac{a^r}{r!} + \beta\frac{b^r}{r!} + \sum_{i=1}^{r-1} \frac{a!b!}{i!(r-i)!} + 2\eps \frac{n^r}{r!}\\
&= (\mathfrak{g}_{d(G_1), \beta}(\frac{a}{n}) + 2\eps) \frac{n^r}{r!}\\
&\leq (d(G_1) \oplus_r \beta + 2\eps)\frac{n^r}{r!}, \textrm{by Lemma \ref{lem:max2}},\\
&\leq ((\alpha - c_2)\oplus_r \beta+2\eps)\frac{n^r}{r!}\\
&\leq \left(\alpha \oplus_r \beta + 2\eps - c_2\left(\frac{\sqrt[r-1]{1-\beta}}{1+\sqrt[r-1]{1-\beta}}\right)^r\right)\frac{n^r}{r!}, \textrm{by Lemma \ref{lem:monoton}},\\
&\leByRef{eq:eps3} (\alpha \oplus_r \beta - \eps)\frac{n^r}{r!}\\
&\leByRef{eq:n_03} (\alpha \oplus_r \beta)\binom{n}{r},
\end{align*}
a contradiction.
\end{proof}

In fact, the same proof shows that $d(H_1) \geq \alpha -c_2$ and $d(H_2) \geq \beta - c_2$.

Define $S_1 := \{ x \in V(G_1) : d_{G_2}(x) \leq (1 - \eps^{1/2})\binom{b}{r-1}\}$ and $S_2 := \{ x \in V(G_2) : d_{G_1}(x) \leq (1 - \eps^{1/2})\binom{a}{r-1}\}$.

\begin{claim}
\label{clm:clm4}
$|S_1| \leq c_3a$ and $|S_2| \leq c_3 b$.
\end{claim}
\begin{proof}
We only prove that $|S_1| \leq c_3a$, as the other statement is proved similarly.

Note that $\eps\binom{n}{r} \geq |E(H) \setminus E(G)| \geq |S_1|\eps^{1/2}\binom{b}{r-1}$, hence
\begin{align*}
|S_1| &\leq \eps^{1/2} \binom{n}{r}\binom{b}{r-1}^{-1}\\
&\leq \eps^{1/2}\frac{n^r}{r(b-r)^{r-1}}\\
&\leq \eps^{1/2}\frac{n}{r(1 - x_{\alpha, \beta} - c_1 -\frac{r}{n})^{r-1}}, \textrm{ by Claim \ref{clm:clm2}},\\
&\leByRef{eq:n_01} \eps^{1/2} \frac{n}{r(1-x_{\alpha, \beta}-2c_1)^{r-1}}\\
&\leByRef{eq:eps4} c_3(x_{\alpha, \beta} - c_1)n\\
&\leq c_3a, \textrm{ by Claim \ref{clm:clm2}}.
\end{align*}
\end{proof}

Define $G_1' := G_1 \setminus S_1$ and $G_2' := G_2 \setminus S_2$. Furthermore set $H_1' := H_1 \setminus S_1$ and $H_2' := H_2 \setminus S_2$. By \eqref{eq:n_21}, $G_1'$ and $G_2'$ have each at least $n_1$ vertices. Moreover,
\begin{equation}
\label{eq:density1}
d(H_1') \geq d(H_1) - c_3r \geq \alpha - c_2 - c_3r \gByRef{eq:c231} \theta(\mathcal{F}_\alpha(M)) + \eps_0,
\end{equation}
and similarly
\begin{equation}
\label{eq:density2}
d(H_2') \geq \beta - c_2 - c_3r \gByRef{eq:c231} \theta(\mathcal{F}_\beta(M)) + \eps_0.
\end{equation}

\begin{claim}
\label{clm:clm5}
Let $V(G_1') \subseteq U$ and $V(G_2') \subseteq W$ be disjoint sets of vertices in $G$ with the following property: $d_{G_2'}(x) \geq (1 - c_6)\binom{b}{r-1}$ for any $x \in U$, and $d_{G_1'}(x) \geq (1 - c_6)\binom{a}{r-1}$ for any $x \in W$. Then $G[U]$ is $\mathcal{F}_\alpha$-free and $G[W]$ is $\mathcal{F}_\beta$-free.
\end{claim}
\begin{proof}
We only prove that $G[U]$ is $\mathcal{F}_\alpha$-free, as the other statement is proved in a similar way.

Suppose for a contradiction that for some $F \in \mathcal{F}_\alpha$, there is a copy of $F$ inside $G[U]$, which we also denote by $F$.

Recall that by the Addendum to the Rigidity Lemma, (b), $C(F)$ contains a copy $K$ of $F$ such that $C(F) \setminus K$ has a partition $(C_1, C_2)$ valid w.r.t. $\mathcal{F}_\alpha(M)$ and $\mathcal{F}_\beta(M)$. Either of $C_1$ or $C_2$ can be empty, but not both, as then $C(F) \simeq F$, contradicting the fact that $G$ is $\mathcal{F}_{\alpha, \beta}$-free. By our choice of $\delta$, \eqref{eq:density1}, \eqref{eq:density2} and the fact that $n_1 \geq n^{*}$, there are at least $\delta v(H_1')^{v(C_1)}$ copies of $C_1$ in $H_1'$, and at least $\delta v(H_2')^{v(C_2)}$ copies of $C_2$ in $H_2'$. By Lemma \ref{lem:remset} and \eqref{eq:n_1C3}, we can find $N \geq \frac{\delta^2}{2}v(H_1')^{v(C_1)}v(H_2')^{v(C_2)}$ embeddings of $C_1 \times C_2$ in $(H_1' \setminus V(F)) \times H_2'$, mapping $C_1$ into $H_1'$ and $C_2$ into $H_2'$.

As $C(F)$ is not a subgraph of $G$, for each of the above embeddings $f : C_1 \times C_2 \rightarrow (H_1' \setminus V(F)) \times H_2'$, either one of the edges in $f(C_1 \times C_2)$ is in $E(H) \setminus E(G)$, or one of the edges intersecting $V(F)$ in one vertex and $f(C_2)$ in $r-1$ vertices is in $E(H) \setminus E(G)$.

Suppose first that in at least $\frac{N}{2}$ of the embeddings $f$, one of the edges in $f(C_1 \times C_2)$ is in $E(H) \setminus E(G)$.

Every edge $e$ in $f(C_1 \times C_2)$ intersects $H_1'$ in some $s_1(e) \geq 0$ vertices, and $H_2'$ in some $s_2(e) \geq 0$ vertices, so that $s_1(e) + s_2(e) = r$. Thus there are $s_1$ and $s_2$ such that in at least $\frac{N}{2(r+1)}$ of the embeddings $f$, there is an edge in $f(C_1 \times C_2)$ intersecting $H_1'$ in $s_1$ vertices, and $H_2'$ in $s_2$ vertices, and this edge is in $E(H) \setminus E(G)$.

We count the number of such edges. It is at least
\begin{align*}
\frac{N}{2(r+1)a^{v(C_1) - s_1}b^{v(C_2)-s_2}} 
&\geq \frac{\delta^2}{4(r+1)}(1-c_3)^{v(C_1)+v(C_2)}\frac{a^{v(C_1)}b^{v(C_2)}}{a^{v(C_1) - s_1}b^{v(C_2)-s_2}} \\
&= \frac{\delta^2}{4(r+1)}(1-c_3)^{v(C_1)+v(C_2)}a^{s_1}b^{s_2}\\
&\geq \frac{\delta^2}{4(r+1)}(1-c_3)^{M_{\alpha, \beta}}(x_{\alpha, \beta} - c_1)^{r}(1 - x_{\alpha, \beta} - c_1)^{r}n^r, \textrm{as $v(C(F)) \leq M_{\alpha, \beta}$,}\\
&\gByRef{eq:eps5} \eps\frac{n^r}{r!}.
\end{align*}
Thus $|E(H)\setminus E(G)| > \eps\binom{n}{r}$, a contradiction.

Consequently in at least $\frac{N}{2}$ of the embeddings $f$, one of the edges intersecting $V(F)$ in one vertex and $f(C_2)$ in $r-1$ vertices is in $E(H) \setminus E(G)$. Then for some $x \in V(F)$ we have
\begin{align*}
\binom{b}{r-1} - d_{G_2'}(x) &\geq \frac{N}{2v(F)a^{v(C_1)}b^{v(C_2)-(r-1)}}\\
&\geq \frac{\delta^2}{4v(F)}(1-c_3)^{v(C_1)+v(C_2)}\frac{a^{v(C_1)}b^{v(C_2)}}{a^{v(C_1)}b^{v(C_2)-(r-1)}}\\
&\geq \frac{\delta^2}{4M_{\alpha, \beta}}(1-c_3)^{M_{\alpha, \beta}}b^{r-1},
\end{align*}
a contradiction with \eqref{eq:eps6} and our assumption that $d_{G_2'}(x) \geq (1-c_6)\binom{b}{r-1}$.
\end{proof}

Every vertex $x \in V(G_1')$ has degree
\begin{equation*}
d_{G_2'}(x) \geq (1-\eps^{1/2})\binom{b}{r-1} - |S_2|\binom{b-1}{r-2} \geq (1-\eps^{1/2} - c_3r)\binom{b}{r-1},
\end{equation*}
and similarly for every $x \in V(G_2')$ we have $d_{G_1'}(x) \geq (1-\eps^{1/2} - c_3r)\binom{a}{r-1}$. As $\eps^{1/2} +c_3r < c_6$ by \eqref{eq:eps7}, it follows from Claim~\ref{clm:clm5} that $G_1'$ is $\mathcal{F}_\alpha$-free and $G_2'$ is $\mathcal{F}_\beta$-free.

\begin{claim}
\label{clm:clm6}
For any $x \in S_1 \dot\cup S_2$, either $d_{G_1'}(x) \leq (\alpha + c_4)\binom{a}{r-1}$ or $d_{G_2'}(x) \leq (\beta + c_4)\binom{b}{r-1}$.
\end{claim}
\begin{proof}
Let $x \in S_1 \dot\cup S_2$ arbitrary and assume for a contradiction that $d_{G_1'}(x) > (\alpha + c_4)\binom{a}{r-1}$ and $d_{G_2'}(x) > (\beta + c_4)\binom{b}{r-1}$.

Then by Lemma \ref{lem:add_vertex}, \eqref{eq:n_1V} and the fact that $G_1'$ is $\mathcal{F}_\alpha$-free, there exist $P \in \mathcal{F}_\alpha$, $v \in V(P)$ and a copy $K_1$ of $P \setminus v$ in $G_1'$ such that $K_1$ and $x$ form a copy of $P$ in $G$. Similarly, there exist $Q \in \mathcal{F}_\beta$, $w \in V(Q)$ and a copy $K_2$ of $Q \setminus w$ in $G_2'$ such that $K_2$ and $x$ form a copy of $Q$ in $G$.

Recall the description of $C(P, Q, v, w)$ given in Lemma \ref{lem:rigid}, (B). According to this description, $C(P, Q, v, w)$ contains a copy of $P \setminus v$, which we also denote by $K_1$, and a copy of $Q \setminus w$, which we also denote by $K_2$, and a vertex $z$, such that $C(P, Q, v, w) \setminus (K_1 \cup K_2 \cup \{z\})$ has a partition $(C_1, C_2)$ valid w.r.t. $\mathcal{F}_\alpha(M)$ and $\mathcal{F}_\beta(M)$. Again one of $C_1$ or $C_2$ can be empty, but not both, as then $C(P, Q, v, w) \subseteq G$, a contradiction with the fact that $G$ is $\mathcal{F}_{\alpha, \beta}$-free.

By our choice of $\delta$, \eqref{eq:density1}, \eqref{eq:density2} and the fact that $n_1 \geq n^{*}$, there are at least $\delta v(H_1')^{v(C_1)}$ copies of $C_1$ in $H_1'$, and at least $\delta v(H_2')^{v(C_2)}$ copies of $C_2$ in $H_2'$. By Lemma \ref{lem:remset} and \eqref{eq:n_1C3}, we can find $N \geq \frac{\delta^2}{4}v(H_1')^{v(C_1)}v(H_2')^{v(C_2)}$ embeddings $f : C_1 \times C_2 \rightarrow (H_1' \setminus V(K_1)) \times (H_2' \setminus V(K_2))$.

Together with $K_1, K_2$ and $x$, any such embedding would potentially form a copy of $C(P, Q, v, w)$ in $G$. Therefore as in the proof of Claim \ref{clm:clm5}, we distinguish two cases.

First suppose there are $s_1$ and $s_2$ such that in at least $\frac{N}{2(r+1)}$ of the embeddings $f$, one of the edges in $f(C_1 \times C_2)$ is in $E(H) \setminus E(G)$, and intersects $H_1'$ in $s_1$ vertices, and $H_2'$ in $s_2$ vertices. A similar count to that in Claim \ref{clm:clm5} gives a contradiction.

Consequently in at least $\frac{N}{2}$ of the embeddings $f$, one of the edges intersecting $K_1$ in one vertex and $f(C_2)$ in $r-1$ vertices, or one of the edges intersecting $K_2$ in one vertex and $f(C_1)$ in $r-1$ vertices, is in $E(H) \setminus E(G)$.

Thus w.l.o.g. for some $y \in V(K_1)$ we obtain
\begin{align*}
\binom{b}{r-1} - d_{G_2'}(y) &\geq \frac{N}{4v(P)a^{v(C_1)}b^{v(C_2)-(r-1)}}\\
&\geq \frac{\delta^2}{16M_{\alpha, \beta}}(1-c_3)^{M_{\alpha, \beta}}b^{r-1}\\
&\gByRef{eq:eps6} c_6\binom{b}{r-1},
\end{align*}
a contradiction with the fact that $y \notin S_1$ and hence $d_{G_2'}(y) \geq (1-\eps^{1/2} - c_3r)\binom{b}{r-1} > (1-c_6)\binom{b}{r-1}$.
\end{proof}

Now let $U := \{x \in S_1 \dot\cup S_2 : d_{G_1'}(x) \leq (\alpha + c_4)\binom{a}{r-1}\}$ and $W := S_1\dot\cup S_2 - U$. By Claim \ref{clm:clm6}, any $x \in W$ has $d_{G_2'}(x) \leq (\beta + c_4)\binom{b}{r-1}$.

\begin{claim}
\label{clm:clm7}
For any $x \in  U, d_{G_2'}(x) \geq (1-c_6)\binom{b}{r-1}$ and for any $x \in W, d_{G_1'}(x) \geq (1-c_6)\binom{a}{r-1}$.
\end{claim}
\begin{proof}
We only prove the claim for $x \in W$, as the other statement is similar.

First note the following identity:
\begin{equation}
\label{eq:deg_oplus}
\alpha \oplus_r \beta - \beta(1 - x_{\alpha, \beta})^{r-1} - (r-1)! \sum_{j=1}^{r-2} \frac{x_{\alpha, \beta}^j (1-x_{\alpha, \beta})^{r-1-j}}{j!(r-1-j)!} = x_{\alpha, \beta}^{r-1}.
\end{equation}

Indeed, recall that
\begin{equation*}
\alpha \oplus_r \beta = \mathfrak{g}_{\alpha, \beta}(x_{\alpha, \beta}) = 1 - (1-\alpha)x_{\alpha, \beta}^r - (1-\beta)(1-x_{\alpha, \beta})^r.
\end{equation*}
Also
\begin{equation*}
(r-1)!\sum_{j=1}^{r-2}\frac{x_{\alpha, \beta}^j (1-x_{\alpha, \beta})^{r-1-j}}{j!(r-1-j)!} = 1 - x_{\alpha, \beta}^{r-1} - (1-x_{\alpha, \beta})^{r-1}.
\end{equation*}
Hence
\begin{align*}
\alpha \oplus_r \beta &- \beta(1 - x_{\alpha, \beta})^{r-1} - (r-1)! \sum_{j=1}^{r-2} \frac{x_{\alpha, \beta}^j (1-x_{\alpha, \beta})^{r-1-j}}{j!(r-1-j)!}\\
&= 1 - (1-\alpha)x_{\alpha, \beta}^r - (1-\beta)(1-x_{\alpha, \beta})^{r} -\beta(1-x_{\alpha, \beta})^{r-1} - 1 + x_{\alpha, \beta}^{r-1} + (1-x_{\alpha, \beta})^{r-1}\\
&= x_{\alpha, \beta}^{r-1} - (1-\alpha)x_{\alpha, \beta}^r  + (1-\beta)(1-x_{\alpha, \beta})^{r-1}x_{\alpha, \beta}\\
&= x_{\alpha, \beta}^{r-1} - \frac{(1-\alpha)(1-\beta)}{(\sqrt[r-1]{1-\alpha} + \sqrt[r-1]{1-\beta})^{r-1}}x_{\alpha,\beta} + \frac{(1-\alpha)(1-\beta)}{(\sqrt[r-1]{1-\alpha} + \sqrt[r-1]{1-\beta})^{r-1}}x_{\alpha,\beta}\\
&= x_{\alpha, \beta}^{r-1},
\end{align*}
proving \eqref{eq:deg_oplus}.

By \eqref{eq:n_01}, Lemma \ref{lem:min_degree} and our assumption that $\pi(\mathcal{F}_{\alpha, \beta}) \geq \alpha \oplus_r \beta$, we see that
\begin{equation}
\label{eq:min_deg}
d_G(x) \geq (\alpha \oplus_r \beta - c_1)\binom{n-1}{r-1}
\end{equation}
for any $x \in V(G)$.

Now let $x \in W$ arbitrary. Using \eqref{eq:min_deg} and the fact that $|S_1 \dot\cup S_2| \leq c_3 n$ we obtain
\begin{align*}
d_{G_1'}(x) &\geByRef{eq:n_01} (\alpha \oplus_r \beta - c_1)\binom{n-1}{r-1} - (\beta + c_4)\binom{b}{r-1} - \sum_{j=1}^{r-2}\binom{a}{j}\binom{b}{r-1-j} - c_3r\binom{n-1}{r-1}\\
&\begin{aligned}
\geByRef{eq:n_03} (\alpha \oplus_r \beta - 2c_1 &-(\beta + c_4)(1-x_{\alpha, \beta}+c_1)^{r-1}\\
&-(r-1)!\sum_{j=1}^{r-2}\frac{(x_{\alpha, \beta}+c_1)^{j}(1-x_{\alpha,\beta}+c_1)^{r-1-j}}{j!(r-1-j)!} - c_3r)\frac{n^{r-1}}{(r-1)!}
\end{aligned}\\
&\begin{aligned}
\geq (\alpha \oplus_r \beta - \beta(1-x_{\alpha, \beta})^{r-1} &-(2+2^{r-1})c_1 - c_4 - c_3r\\
&-(r-1)!\sum_{j=1}^{r-2}\frac{x_{\alpha, \beta}^j (1-x_{\alpha, \beta})^{r-1-j} + 3 \cdot 2^{r-1}c_1}{j!(r-1-j)!})\frac{n^{r-1}}{(r-1)!}
\end{aligned}\\
&\geByRef{eq:c5} (\alpha \oplus_r \beta - \beta(1-x_{\alpha, \beta})^{r-1} - (r-1)!\sum_{j=1}^{r-2}\frac{x_{\alpha, \beta}^j (1-x_{\alpha, \beta})^{r-1-j}}{j!(r-1-j)!} - c_5)\frac{n^{r-1}}{(r-1)!}\\
&\eqByRef{eq:deg_oplus} (x_{\alpha, \beta}^{r-1} - c_5)\frac{n^{r-1}}{(r-1)!}\\
&\geByRef{eq:c6} (1-c_6)(x_{\alpha, \beta}+c_1)^{r-1}\frac{n^{r-1}}{(r-1)!}\\
&\geq (1-c_6)\binom{a}{r-1}.
\end{align*}
This proves the claim.
\end{proof}

Define $U' := V(G_1') \cup U$ and $W' := V(G_2') \cup W$. Then by Claims \ref{clm:clm5} and \ref{clm:clm7}, $G[U']$ is $\mathcal{F}_\alpha$-free and $G[W']$ is $\mathcal{F}_\beta$-free. But $(U', W')$ is a partition of $G$. Setting $a':=|U'|$ and $b':=|W'|$, we obtain
\begin{align*}
e(G) &\leq e(G[U']) + e(G[V']) + \sum_{j=1}^{r-1}\binom{a'}{j}\binom{b'}{r-j}\\
&\leq \frac{1}{r!}(\alpha a'^{r} + \beta b'^{r} + r!\sum_{j=1}^{r-1}\frac{a'^jb'^{r-j}}{j!(r-j)!}+\frac{\zeta n^r}{2}), \quad \textrm{by our choice of $n_1$},\\
&\leq (\alpha \oplus_r \beta + \frac{\zeta}{2})\frac{n^r}{r!}\\
&\leByRef{eq:n_03} (\alpha \oplus_r \beta + \zeta)\binom{n}{r}.
\end{align*}
This finishes the proof of the Collapsing Lemma.
\end{proof}

\subsection{\normalsize End of the proof}

We are now ready to finish the proof of Theorem \ref{thm:local}. There is only one further ingredient that we need.

\begin{theorem}[Strong Removal Lemma, \cite{Rodl09}]
For every family $\mathcal{F}$ of $r$-graphs and any $\eps > 0$ there exist $\delta, m$ and $n_R$ such that the following holds. If $G$ is any $r$-graph on $n \geq n_R$ vertices which contains at most $\delta n^{v(F)}$ copies of any $r$-graph $F \in \mathcal{F}$ with $v(F) \leq m$, then $G$ can be made $\mathcal{F}$-free by removing at most $\eps\binom{n}{k}$ edges.
\end{theorem}

The Removal Lemma for hypergraphs is a deep and rather recent result. Its origins can be traced back to the 70s, in the (now famous) Triangle Removal Lemma of Ruzsa and Szemer\'edi \cite{ruzsa76}, though it was only in the last decade that a suitable version for hypergraphs was obtained, independently by Gowers \cite{Gowers07} and by Nagle, R{\"o}dl, Schacht and Skokan (\cite{Nagle06}, \cite{Rodl04}, \cite{Rodl06}). Subsequently generalizations and other versions were proved. We remark that the original Removal Lemma is stated in terms of a single hypergraph; we crucially need here a version applicable to an \textit{infinite} family of hypergraphs.

\begin{proof}[Proof of Theorem \ref{thm:local}]
Let $\alpha, \beta \in \Pi_{\fin}^{(r)}$. In view of Lemma \ref{lem:local_inf}, to finish the proof we only need to show that $\alpha \oplus_r \beta \in \Pi_{\fin}^{(r)}$. Clearly we may assume that $\alpha, \beta \neq 1$.

Choose a finite family of $r$-graphs $\mathcal{F}_\alpha$ with $\pi(\mathcal{F}_\alpha)=\alpha$. By Lemma \ref{lem:minimal}, we may assume that $\mathcal{F}_\alpha$ is minimal. As $\alpha \neq 1$, $\mathcal{F}_\alpha$ is non-empty.

Similarly we can choose a finite non-empty minimal family of $r$-graphs $\mathcal{F}_\beta$ with $\pi(\mathcal{F}_\beta) = \beta$.

On input $\mathcal{F}_\alpha$ and $\mathcal{F}_\beta$, the Collapsing Lemma gives a positive integer $M := M_{\alpha, \beta}$. Choose and fix arbitrary $M$-closures $\mathcal{F}_\alpha(M)$ and $\mathcal{F}_\beta(M)$. The Collapsing Lemma now gives us a finite family of $r$-graphs $\mathcal{F}_{\alpha, \beta}$ and an $\eps > 0$.

Let $\{G_n^1\}_{n \geq 1}$ be any sequence of $\mathcal{F}_\alpha(M)$-free $r$-graphs with $v(G_n^1)=n$ and $d(G_n^1) \rightarrow \alpha$. Such a sequence exists, even in the case $\alpha = 0$, as any $r$-graph in $\mathcal{F}_\alpha(M)$ has at least one edge. Similarly, let $\{G_n^2\}_{n \geq 1}$ be any sequence of $\mathcal{F}_\beta(M)$-free $r$-graphs with $v(G_n^2)=n$ and $d(G_n^2) \rightarrow \beta$.

Consider the sequence $\{G_n\}_{n \geq 1}$ with $G_n := G^1_{x_{\alpha, \beta}n} \times G^2_{(1-x_{\alpha, \beta})n}, n \geq 1$ (we disregard lower and upper integer parts here, as it does not affect our proof). Then $d(G_n)$ converges to $\alpha \oplus_r \beta$. Moreover $G_n$ is $\mathcal{F}_{\alpha, \beta}$-free by the first part of the Collapsing Lemma, for any $n \geq 1$.

Define $\mathcal{F}_{\infty} := \{F : F \not\subseteq G_n, \forall n \geq 1\}$. Then $\mathcal{F}_{\alpha, \beta} \subseteq \mathcal{F}_{\infty}$. Apply the Strong Removal Lemma to $\mathcal{F}_{\infty}$ and $\frac{\eps}{2}$ to obtain $\delta$ (which we disregard), $m$ and $n_R$.

Finally, set $\mathcal{F}_m := \mathcal{F}_{\alpha, \beta} \cup \{F \in \mathcal{F}_{\infty} : v(F) \leq m\}$.

We claim $\pi(\mathcal{F}_m)=\alpha \oplus_r \beta$.

Clearly $\pi(\mathcal{F}_m) \geq \alpha \oplus_r \beta$, as the sequence $\{G_n\}_{n \geq 1}$ shows.

Let $\zeta > 0$ arbitrary. We show $\pi(\mathcal{F}_m) \leq \alpha \oplus_r \beta + \zeta$.

The Collapsing Lemma gives us an $n_0 \geq 1$. Let $G$ be any maximum $\mathcal{F}_m$-free $r$-graph on $n \geq \max\{n_0, n_R\}$ vertices. Then $G$ can be made $\mathcal{F}_{\infty}$-free by removing at most $\frac{\eps}{2}\binom{n}{r}$ edges. Let $G'$ be the resulting $r$-graph. As $G' \notin \mathcal{F}_{\infty}$, there exists $k \geq 1$ such that $G' \subseteq G_k$. Let $H$ be the subgraph of $G_k$ isomorphic with $G'$. Now $H' := G_k[V(H)]$ is also $\mathcal{F}_m$-free and hence $e(G) \geq e(H') \geq e(G')$. Thus $|E(H') \setminus E(H)| \leq \frac{\eps}{2}\binom{n}{r}$. Consequently $G$ is $\eps$-close to $H'$. But $H'$, being an induced subgraph, is of the form $H_1 \times H_2$, with $H_1$ an $\mathcal{F}_\alpha(M)$-free $r$-graph, and $H_2$ an $\mathcal{F}_\beta(M)$-free $r$-graph.

By the Collapsing Lemma, $G$ has at most $(\alpha \oplus_r \beta + \zeta)\binom{n}{r}$ edges. Thus $\pi(\mathcal{F}_m) \leq \alpha \oplus_r \beta + \zeta$. As $\zeta$ was arbitrary, the proof is finished.
\end{proof}

We briefly highlight some of the difficulties we had to overcome in the above proof.

The general strategy was taken from the proof of Theorem 3 in \cite{pikhurko12}. While the Removal Lemma allows us to force any maximum $\mathcal{F}_{\alpha, \beta}$-free $r$-graph to be close to the desired structure by adding some more forbidden $r$-graphs, it can not be made arbitrarily close. This requires the proof of a Collapsing Lemma.

The most serious obstacle appears when we pass from the sequence $\{G_n\}_{n \geq 1}$ to the induced subgraph $H'$. While $H'$ is a subgraph of some $G_k$, the number of vertices of $H'$ is not (in any way) bounded from below by the number of vertices of $G_k$. Thus properties of graphs in the sequence $\{G_n\}_{n \geq 1}$ need not pass to $H'$. In particular, if $G_k$ contains $\Omega(v(G_k)^{v(F)})$ copies of some $r$-graph $F$, $H'$ can very well have only a few such copies, or none at all. In order to overcome this obstacle we used information hidden in the function $\pi$. This is reflected in the Rigidity Lemma, which is not stated for a particular graph sequence, but more generally in terms of two families of $r$-graphs.

\section{\normalsize Explicit irrational densities}

We now prove Corollaries~\ref{cor:irrat1} and~\ref{cor:irrat2}, which offer explicit examples of irrational Tur\'an densities. We first need a result from number theory.

\begin{theorem}[\cite{carr09}]
\label{thm:LinIndep}
Let $r_i > 0$ be roots of rationals (i.e. $r_i^{n_i} \in \mathbb{Q}$ with $n_i \in \mathbb{N}$) for $i$ in a finite indexing set $I$. Suppose
\begin{equation*}
\sum_{i \in I}q_i r_i = q \in \mathbb{Q}
\end{equation*}
for positive rationals $q_i$. Then each $r_i$ is rational.
\end{theorem}

\begin{proof}[Proof of Corollary~\ref{cor:irrat1}]
Let $r \geq 3$. As $\lambda(e) = \frac{r!}{r^r}$, where $e$ is the $r$-edge, by Theorem~\ref{thm:pikfinite} we have $\frac{r!}{r^r} \in \Pi_{\fin}^{(r)}$. Hence by Theorem \ref{thm:local},
\begin{equation*}
\frac{r!}{r} \oplus_r 0 = 1 - \frac{r^{r-1} - (r-1)!}{\left(r+\sqrt[r-1]{r^{r-1}-(r-1)!}\right)^{r-1}} \in \Pi_{\fin}^{(r)}.
\end{equation*}
This number is rational if and only if $q := \left(r+\sqrt[r-1]{r^{r-1}-(r-1)!}\right)^{r-1}$ is rational.

Assume for a contradiction that $q$ is rational. Let $r_i := (r^{r-1}-(r-1)!)^{\frac{i}{r-1}}$ for $0 \leq i \leq r-1$. Then
\begin{equation*}
q = \sum_{i=0}^{r-1}\binom{r-1}{i}r^{r-1-i} r_i.
\end{equation*}
All $r_i$ are positive roots of rationals. Hence by Theorem \ref{thm:LinIndep}, we obtain that $r_i$ is rational for all $0 \leq i \leq r-1$. In particular, $r_1$ is a natural number.

As $r \geq 3$, we have $r_1 > 1$. So we can find a prime divisor $p | r_1$. Then $p^{r-1} | r^{r-1}-(r-1)!$. Hence $1 < p < r$ and so $p | (r-1)!$.

Thus $p|r^{r-1}$. But then $p | r$ and so certainly $p^{r-1} | r^{r-1}$. As $p^{r-1}$ divides $r^{r-1}-(r-1)!$, it must also divide $(r-1)!$. But it is well known that for any prime $p$, the power of $p$ dividing $(r-1)!$ is
$$\left\lfloor \frac{r-1}{p} \right\rfloor + \left\lfloor \frac{r-1}{p^2} \right\rfloor + \ldots < \frac{r-1}{p-1} \leq r-1.$$
This is a contradiction, completing the proof.
\end{proof}

For the proof of Corollary~\ref{cor:irrat2} we shall need the following result of Sidorenko.
\begin{theorem}[Sidorenko, \cite{sidorenko92}]
\label{thm:sid}
$1 - \frac{1}{2^p} \in \Pi_{\fin}^{(2k)}$ for any $k, p \geq 1$.
\end{theorem}
\begin{proof}[Proof of Corollary~\ref{cor:irrat2}]
Let $r \geq 4$ even. By Theorem \ref{thm:sid}, $\frac{1}{2} \in \Pi_{\fin}^{(r)}$. Consequently by Theorem \ref{thm:local},
\begin{equation*}
\frac{1}{2} \oplus_r 0 = 1 - \frac{1}{(1+\sqrt[r-1]{2})^{r-1}} \in \Pi_{\fin}^{(r)}.
\end{equation*}
This number is rational if and only if $(1+\sqrt[r-1]{2})^{r-1}$ is rational.

However, $x_0 := \sqrt[r-1]{2}$ has the minimal polynomial $f(x) = x^{r-1} - 2$ ($f(x)$ is irreducible by Eisenstein's criterion). Consequently if $(1+x_0)^{r-1}$ equals some rational $\frac{s}{t}$, then $f(x)$ must divide $t(1+x)^{r-1} - s$. Then $tf(x) = t(1+x)^{r-1} - s$, which is not possible as $r \geq 4$. Thus $(1+x_0)^{r-1}$ is irrational, completing the proof.
\end{proof}

\section{\normalsize Towards a semiring structure}

As we have seen in the Introduction, the set $\Pi_{\infty}^{(2)}$ has a semiring structure given by the two operations $\oplus_2$ and $\otimes_2$. We have already successfully generalized the operation $\oplus_2$ to any $r \geq 2$. It is thus natural to try to do the same with $\otimes_2$.

Unfortunately we will not be so lucky this time.

As in the case of $\oplus_r$, we must find a corresponding operation on $r$-graphs. Recall that for any $\alpha, \beta \in \Pi_{\infty}^{(2)}$ we have that $1- \alpha \otimes_2 \beta = (1-\alpha)(1-\beta)$. There would be many advantages if this would hold for any $r$, in particular $\otimes_2$ is distributive over $\oplus_r$ for any $r$, and we would obtain a semiring structure on $\Pi_{\infty}^{(r)}$ as desired. As far as we can see, there is only one natural construction associated to this operation (we keep in mind the concrete examples given by $r=2$ and graph cliques).

Let $G$ and $H$ be two $r$-graphs. We define an $r$-graph $G \otimes H$ in the following way.

Assume w.l.o.g. that $G$ has vertex set $[n]$. The vertex set of $G \otimes H$ consists of $n$ disjoint copies $V_1, \ldots, V_n$ of the vertex set of $H$. If $v$ is a vertex of $H$, we let $v_i \in V_i$ be its $i$-th copy. We add the following edges to $G \otimes H$.

For all $h=(v^1, \ldots, v^r) \in E(H)$, we add all edges $f$ with $|f \cap \{v^t_{1}, \ldots, v^t_{n}\}| = 1, 1 \leq t \leq r$. Furthermore for all $e = (i_1, \ldots, i_r) \in E(G)$, we add all edges $f$ with $|f \cap V_{i_j}| = 1, 1 \leq j \leq r$. No other edges are added.

This is a generalization of the strong product of graphs to uniform hypergraphs.

Note that if $r=2$, $G$ is an $n$-clique\footnote{i.e. a complete $2$-graph on $n$ vertices.} and $H$ is an $m$-clique then $G \otimes H$ is an $mn$-clique. This corresponds to our objective and hence we would like to prove the following.

\begin{target}
\label{tar:target}
For any two $r$-graphs $G$ and $H$ we have $\lambda(G \otimes H) = \lambda(G) + \lambda(H) - \lambda(G)\lambda(H)$.
\end{target}

Target \ref{tar:target} would imply via continuity that $\Pi_\infty^{(r)}$ is closed under $\otimes_2$.

It is easy to show that $\lambda(G \otimes H) \geq \lambda(G) + \lambda(H) - \lambda(G)\lambda(H)$. Indeed, assume $G$ has vertex set $[n]$ and $H$ has vertex set $[m]$. We identify the vertex set of $G \otimes H$ with $[nm]$: the vertex set of the $i$-th copy of $H$ runs from $(i-1)m+1$ to $im$. If $\mathbf{a} \in \Delta_n$ is an optimal vector for $G$ and $\mathbf{b}$ is an optimal vector for $H$, then considering the vector $\mathbf{c} \in \Delta_{nm}$ with $c_{(i-1)m+j} := a_ib_j, 1 \leq i \leq n, 1 \leq j \leq m,$ we see that $\lambda(G \otimes H) \geq \lambda(G) + \lambda(H) - \lambda(G)\lambda(H)$, as claimed.

One can easily check that equality holds true for $r=2$.

But for $r \geq 3$ Target \ref{tar:target} is false.

\subsection{\normalsize A permanent d\'etour}

Let us look at the special case when both $G$ and $H$ are one edge $r$-graphs. Then $\lambda(G) = \lambda(H) = \frac{r!}{r^r}$ and we would like that $\lambda(G \otimes H) = \frac{2r!}{r^r} - \frac{(r!)^2}{r^{2r}}$. This claim is highly non-trivial; it is equivalent to a statement known as Dittert's conjecture.

First let us state the former van der Waerden conjecture, now a theorem.
\begin{theorem}[Egorychev, Falikman, 1981]
\label{thm:waerden}
The minimum permanent among all $n \times n$ doubly stochastic matrices is $\frac{n!}{n^n}$, and is achieved only by the matrix with all entries equal to $1/n$.
\end{theorem}
Theorem \ref{thm:waerden} was conjectured by van der Waerden in 1926. After attracting a lot of interest and a series of partial results, it was finally proved independently by Egorychev and Falikman in 1981, and it is still considered a milestone result in combinatorics, with many applications across the field. A different proof was recently found by Gurvits \cite{Gurvits08}.

Several other similar conjectures have been made over time. The following conjecture is due to E. Dittert.
\begin{conjecture}[Dittert, 1983, \cite{Minc83}]
\label{conj:dittert}
Let $A$ be a non-negative $n \times n$ matrix with row sums $r_1, \ldots, r_n$ and column sums $c_1, \ldots, c_n$. Suppose $\sum_{i=1}^n r_i = \sum_{i=1}^n c_i = 1$. Then the function
\begin{equation*}
\psi(A) := \prod_{i=1}^n r_i + \prod_{i=1}^n c_i - \textrm{\upshape{per}}(A)
\end{equation*}
has the maximum $\frac{2}{n^n} - \frac{n!}{n^{2n}}$, and it is achieved only by the matrix $J_n$ with all entries equal to $1/n^2$.
\end{conjecture}

Conjecture \ref{conj:dittert} clearly implies Theorem \ref{thm:waerden}. There is substantial evidence towards Conjecture \ref{conj:dittert}. It was proved for $n=2$ by Sinkhorn \cite{Sinkhorn84} and for $n=3$ by Hwang \cite{hwang87}. Hwang further showed in \cite{Hwang86} that if the $\psi$-maximising matrix is positive then it must equal $J_n$, and that $J_n$ is a strict local maximum of $\psi$. Other partial results were obtained by Cheon and Yoon \cite{Cheon06} and Cheon and Wanless \cite{Cheon12}. Most importantly Cheon and Wanless \cite{Cheon12} showed that the maximum value of $\psi$ is exponentially close to the conjectured value.

\begin{theorem}[Cheon-Wanless, \cite{Cheon12}]
\label{thm:wanless}
For any non-negative $n \times n$ matrix $A$ with the sum of all elements equal to $1$ we have $\psi(A) < \psi(J_n) + O(n^{4-n}e^{2n})$.
\end{theorem}

Aside from the unicity of the maximum, it is easy to see that Conjecture \ref{conj:dittert} is equivalent (with $n=r$) to our claim $\lambda(e \otimes e) = \frac{2r!}{r^r} - \frac{(r!)^2}{r^{2r}}$, where $e$ is an $r$-edge. Thus any proof of this claim must give a proof of the notoriously hard van der Waerden conjecture.

Let us now note that Target \ref{tar:target} implies a much stronger statement. Write $e^{\otimes k}$  for $e \otimes \ldots \otimes e$ ($k$ times). Then we would like that
\begin{equation}
\label{eq:general}
\lambda(e^{\otimes k}) = \binom{k}{1}\frac{r!}{r^r} - \binom{k}{2}\frac{(r!)^2}{r^{2r}} + \ldots + (-1)^{k-1}\binom{k}{k}\frac{(r!)^{k}}{r^{kr}},
\end{equation}
To see this, imagine the $r$-graph $e^{\otimes k}$ as an $r \times r \times \ldots \times r$ $k$-dimensional matrix. We want that $\lambda(e^{\otimes k})$ equals $p_{e^{\otimes k}}(\frac{1}{r^k}, \ldots, \frac{1}{r^k})$. We evaluate this polynomial by inclusion-exclusion: first project onto a single coordinate and sum up all the terms after that coordinate; then consider two coordinates and subtract the terms that were added twice etc. Equivalently one can expand the conjectured identity $1 - \lambda(e^{\otimes k}) = (1-\lambda(e))^k$.

While studying hashing, Hajek made the following conjecture.
\begin{conjecture}[Hajek, 1987, \cite{Hajek87}]
\label{conj:hajek}
Let $k$ and $n$ be positive integers and for $1 \leq j \leq k$, let $S_j$ denote the collection of subsets $L$ of $[n]^k$ such that $L$ has cardinality $n$ and no two distinct elements of $L$ have the same $j$ coordinate. Let $S := \cup S_j$ and define the multinomial $F_{n, k}(x) := \sum_{L \in S} \prod_{i \in L} x_i$. Then on $\Delta_{n^k}$, $F_{n, k}$ attains its maximum only at $(\frac{1}{n^k}, \ldots, \frac{1}{n^k})$.
\end{conjecture}

Aside from the unicity of the maximum, Conjecture \ref{conj:hajek} is clearly equivalent to \eqref{eq:general} with $n=r$.

Unfortunately Conjecture \ref{conj:hajek} was shown to be false for $n = 3$ and $k = 4$ by K{\"o}rner and Marton \cite{Korner88}. This in particular disproves Target \ref{tar:target}. The counterexample is a construction using the tetra-code, a self-dual code in $\mathbb{F}_3^4$. K\"{o}rner and Marton also give an upper bound to (their equivalent notion of) $\lambda(e^{\otimes k})$, using graph entropy.

\begin{theorem}[K\"{o}rner-Marton, \cite{Korner88}]
\label{thm:korner}
For any $r$ and $k$ we have that
\begin{equation*}
\log_2\frac{1}{1-\frac{r!}{r^r}} \leq \frac{1}{k}\log_2 \frac{1}{1-\lambda(e^{\otimes k})} \leq \frac{r!}{r^{r-1}}.
\end{equation*}
\end{theorem}
The lower bound follows from the construction given before. Interestingly, as shown in \cite{Korner88}, 
any improvement on these bounds would most likely give an improvement on the best known bounds for the perfect hashing problem in a special case.

\subsection{\normalsize A conjecture about the set of all Tur\'an densities}

Most of our interest in Target \ref{tar:target} stems from the following.

\begin{proposition}
\label{prop:aboutConj}
If $\Pi_{\infty}^{(r)}$ is closed under $\otimes_2$ for all $r \geq 2$ then $\overline{\cup_{r \geq 2} \Pi_{\fin}^{(r)}} = \overline{\cup_{r \geq 2} \Pi_{\infty}^{(r)}} = [0,1]$.
\end{proposition}
\begin{proof}
As $\Pi_{\fin}^{(r)}$ is dense in $\Pi_{\infty}^{(r)}$ for all $r$, we only need to prove that $\overline{\cup_{r \geq 2} \Pi_{\infty}^{(r)}} = [0,1]$.

Fix $\alpha \in (0, 1)$ and let $\eps > 0$ arbitrary. We prove that there exists $\gamma \in \cup_{r \geq 2} \Pi_{\infty}^{(r)}$ with $|\gamma - \alpha| < \eps$.

Write $\alpha = 1 - \frac{1}{\ell}$ for some real $\ell > 1$. Let $\delta > 0$ such $\delta \ell  < \eps$. As $\frac{r!}{r^r} \in \Pi_{\infty}^{(r)}$ (see the discussion after Theorem \ref{thm:pikfinite}) and $\frac{r!}{r^r} \rightarrow 0$, there exists $r$ and $s$ with $1 \leq s < 1 + \delta$ and $s < \ell$ such that $\beta := 1 - \frac{1}{s} \in \Pi_{\infty}^{(r)}$. Then for any $n \geq 1$, $\beta^{\otimes n} := \beta \otimes_2 \beta \otimes_2 \ldots \otimes_2 \beta$ ($n$ times) is an element of $\Pi_{\infty}^{(r)}$, by assumption. However, $\beta^{\otimes n} = 1 - \frac{1}{s^n}$ by definition.

Choose $n$ such that $s^n \leq \ell \leq s^{n+1}$. Then $n \geq 1$ and
\begin{equation*}
\ell - s^n \leq s^{n+1} - s^n \leq \delta s^n \leq \delta \ell < \eps.
\end{equation*}
Consequently
\begin{equation*}
|\beta^{\otimes n} - \alpha| = \frac{\ell - s^n}{\ell s^n} < \eps,
\end{equation*}
as $\ell s^n \geq 1$. This proves the claim.
\end{proof}

The proof of Proposition~\ref{prop:aboutConj} could still be carried over if Target 1 would only hold for hypergraphs of the form $e^{\otimes k}$, where $e$ is an $r$-edge. However, as we have seen, this is not the case. Nevertheless, this made us propose Conjecture \ref{conj:closure}, which was recently proved by Pikhurko \cite{pikhurko15} in a different way.

\subsection{\normalsize Some more results}

One can lift $\otimes_2$ to $\Pi_{\infty}$, and in this setting the law holds.

More precisely, define a binary operation $\circ$ on the set $\mathbb{R} \times \mathbb{N}$ (which contains $\Pi_{\infty}$) as follows:
\begin{align*}
  \circ \colon (\mathbb{R} \times \mathbb{N}) \times (\mathbb{R} \times \mathbb{N}) &\to \mathbb{R} \times \mathbb{N}\\
  (\alpha, r) \times (\beta, s) \phantom{xx} &\mapsto ((\alpha+\beta-\alpha\beta)\binom{r+s}{r}\frac{r^rs^s}{(r+s)^{r+s}}, r+s).
\end{align*}

By using a similar trick as in Theorem \ref{thm:global} one can prove the following result.

\begin{theorem}
\label{thm:lifted}
$(\Pi_{\infty}, \circ)$ is a commutative cancellative semigroup.
\end{theorem}

The associated construction is the following. If $G$ is an $r$-graph and $H$ is an $s$-graph on disjoint vertex sets, we define $G \circ H$ as the $(r+s)$-multigraph on vertex set $V(G) \cup V(H)$ and edge set
\begin{equation*}
\{e \cup f : e \in E(G), f \in V(H)^{(s)}\} \cup \{e \cup f : e \in V(G)^{(r)}, f \in E(H)\}
\end{equation*}
The proof then proceeds similarly to that of Theorem \ref{thm:global}, and so we shall not present it here. Unfortunately $\circ$ and $*$ do not define a ring structure on $\Pi_{\infty}$.

In fact other relations concerning Tur\'an densities can be obtained, though none seem to define any interesting algebraic structure. As an example, we have the following theorem.

\begin{theorem}
\label{thm:rational_map}
For any $r \geq 2$ define the map $\mathfrak{j} : [0,1] \rightarrow [0,1]$ by $\mathfrak{j}(x) = \left(\frac{r - 1}{r - x}\right)^{r-1}$. Then $\mathfrak{j}(\Pi_{\infty}^{(r)}) \subsetneq \Pi_{\infty}^{(r)}$.
\end{theorem}

Again the construction is the only important step of the proof. For any $r$-graph $G$, define $\mathfrak{j}(G)$ as the $r$-multigraph on vertex set $\{v\} \cup V(G)$ ($v$ is a vertex not belonging to $G$) and edge set $E(G) \cup \{\{v\} \cup e : e \in V(G)^{(r-1)}\}$. Then one can show that $\lambda(\mathfrak{j}(G)) = \mathfrak{j}(\lambda(G))$ and Theorem \ref{thm:rational_map} follows by continuity.

\section{\normalsize Open problems}

Our investigation ends up with several open problems, which we now discuss.

\subsection{\normalsize{The set of all Tur\'an densities}}

In view of Conjecture \ref{conj:closure}, one can ask what is the set $\cup_{r \geq 2} \Pi_{\infty}^{(r)}$. We could not even solve the following.

\begin{problem}
Prove or disprove that $\limsup_{r \rightarrow \infty} \Pi_{\infty}^{(r)} = \cup_{r \geq 2} \Pi_{\infty}^{(r)}$.
\end{problem}

Here the limit is taken under the discrete metric, that is, an element belongs to $\limsup_{r \rightarrow \infty} \Pi_{\infty}^{(r)}$ if and only if it belongs to $\Pi_{\infty}^{(r)}$ for infinitely many $r$. By Theorem \ref{thm:sid}, $1 - \frac{1}{2^p} \in \limsup_{r \rightarrow \infty} \Pi_{\infty}^{(r)}$ for any $p \geq 1$, and to the best of our knowledge no other values from this set have been determined. Moreover Sidorenko's proof of Theorem \ref{thm:sid} does not generalize to other Tur\'an densities (\cite{Keevash03}).

\subsection{\normalsize{Polynomials preserving Tur\'an densities}}

By Theorem \ref{thm:global}, the polynomial $\frac{1}{2^{2r}}\binom{2r}{r}x^2$ takes values in $\Pi_{\infty}^{(2r)}$ when evaluated at an element of $\Pi_{\infty}^{(r)}$. The following question remains open.
\begin{problem}
\label{pr:poly}
For any $r \geq 3$ find a polynomial $f \in \mathbb{Q}[x]$ such that for any Tur\'an density $\alpha$ for $r$-graphs, $f(\alpha)$ is also a Tur\'an density for $r$-graphs.
\end{problem}
For $r=2$ one such polynomial is $2x-x^2$ (indeed, this is nothing else than our rule $\otimes_2$). Moreover an example of a rational function with the required properties is given by Theorem \ref{thm:rational_map}.

\subsection{\normalsize{The algebraic degree of Tur\'an densities}}

As the reader recalls, this paper was started by Question \ref{ques:fox}. We have not been able to resolve it, though $\oplus_r$ prompts the following question.

\begin{problem}
\label{pr:degree}
For some $r \geq 3$, find $\alpha \in \Pi_{\fin}^{(r)}$ algebraic with minimal polynomial of degree greater than $r-1$, or show that none exists.
\end{problem}

\subsection{\normalsize Other finiteness theorems}

In view of Theorem \ref{thm:local} it is natural to ask the following question.

\begin{question}
Do Theorems \ref{thm:global}, \ref{thm:lifted} and \ref{thm:rational_map} have finite counterparts? That is, is $\Pi_{\fin}$ (respectively $\Pi_{\fin}^{(r)}$) closed under the described operations?
\end{question}

I expect that the methods of this paper would suffice to give a positive answer, though I have not pursued this line of inquiry myself.

\subsection{\normalsize{The Hausdorff dimension of $\Pi_{\infty}^{(r)}$}}

Recall the map $\mathfrak{h} : [0,1) \rightarrow [1, +\infty)$ defined by $\mathfrak{h}(x) = \left(\frac{1}{1-x}\right)^{1/(r-1)}$ in Corollary \ref{cor:lebesgue}. It is an isomorphism between $A := \Pi_{\infty}^{(r)}\setminus\{1\}$ and a subsemigroup of $(\mathbb{R}, +)$. As $\mathfrak{h}^{-1}$ is Lipschitz, if $\Pi_{\infty}^{(r)}$ has positive Hausdorff dimension then so does $\mathfrak{h}(A)$. What can we say about $\mathfrak{h}(A)$ in this case?

Recall that a subset of $\mathbb{R}$ is called analytic if it is the continuous image of some Borel set in some Euclidean space $\mathbb{R}^n$.

\begin{proposition}
\label{prop:dense}
Let $\mathbb{G}_r$ be the subgroup of $\mathbb{R}$ generated by $\mathfrak{h}(A)$ under addition. Then $\mathbb{G}_r$ is an analytic set and for any $r \geq 3$, it is dense in $\mathbb{R}$.
\end{proposition}
\begin{proof}
$\mathbb{G}_r$ is generated by an analytic set and so it must be analytic too.

Furthermore for $r \geq 3$, $\mathbb{G}_r$ contains $\mathbb{Z}$ (as it contains $1$) and $n \alpha, n \geq 1$, where $\alpha$ is some irrational number. By Diophantine approximation, $\mathbb{G}_r$ is dense in $[0,1]$, and hence dense in $\mathbb{R}$.
\end{proof}

It was a question of Erd\H os and Volkmann \cite{ErdosVolk66} if there exist subrings of $\mathbb{R}$ which are Borel sets and have Hausdorff dimension strictly between $0$ and $1$. This question was resolved by Edgar and Miller in 2003 (a discrete version was proved independently by Bourgain \cite{Bourgain03}).

\begin{theorem}[Edgar-Miller, 2003, \cite{Edgar03}]
\label{thm:volkmann}
If $E \subseteq \mathbb{R}$ is a subring and a Borel (or analytic) set then either $E$ has Hausdorff dimension $0$ or $E = \mathbb{R}$.
\end{theorem}

This implies the following.

\begin{proposition}
\label{prop:subring}
Suppose $\Pi_{\infty}^{(r)}$ is closed under $\otimes_2$. Then $\mathbb{G}_r$ is a subring of $\mathbb{R}$.  If $\Pi_{\infty}^{(r)}$ has positive Hausdorff dimension then $\mathbb{G}_r = \mathbb{R}$.
\end{proposition}
\begin{proof}
If $\Pi_{\infty}^{(r)}$ is closed under $\otimes_2$ then $\mathfrak{h}(A)$ is a semigroup under real multiplication. As any element of $\mathbb{G}_r$ is of the form $\alpha - \beta$, with $\alpha, \beta \in \mathfrak{h}(A)$, $\mathbb{G}_r$ must be closed under multiplication as well. As $1 \in \mathbb{G}_r$, $\mathbb{G}_r$ is a subring and an analytic set.

If $\Pi_{\infty}^{(r)}$ has positive Hausdorff dimension, $\mathbb{G}_r$ has too, and hence by Theorem \ref{thm:volkmann}, $\mathbb{G}_r = \mathbb{R}$.
\end{proof}

This might help in resolving the following two problems.

\begin{problem}
Is $\Pi_{\infty}^{(r)}$ closed under $\otimes_2$ for $r \geq 3$?
\end{problem}

\begin{problem}
Compute the Hausdorff dimension of $\Pi_{\infty}^{(r)}$ or at least determine if it is zero.
\end{problem}

\subsection{\normalsize{Revisiting the case $r=2$}}

It is a consequence of the Erd\H os-Stone-Simonovits theorem that
\begin{equation}
\label{eq:revisit}
\Pi_{\fin}^{(2)} = \Pi_{\infty}^{(2)} = \{1\} \cup \{1 - \frac{1}{k} : k \geq 1\}.
\end{equation}
Consider the following problem.
\begin{problem}
Find a proof of \eqref{eq:revisit} without relying on the Erd\H os-Stone-Simonovits theorem, and generalize it as much as possible to $r \geq 3$.
\end{problem}
Here is a short proof. By Tur\'an's theorem, $\{1 - \frac{1}{k} : k \geq 1\} \subset \Pi_{\fin}^{(2)}$. Moreover, $\pi(\emptyset) = 1 \in \Pi_{\fin}^{(2)}$. On the other hand, by the results of Brown and Simonovits, $\Pi_{\fin}^{(2)} \subseteq \Pi_{\infty}^{(2)} \subseteq \overline{\Lambda}^{(2)}$.
It is well-known that $\Lambda^{(2)} = \{1\} \cup \{1 - \frac{1}{k} : k \geq 1\}$, and so \eqref{eq:revisit} holds.

This argument is in some sense unsatisfactory, as it relies on the exact computation of $\Lambda^{(2)}$, a feat which we can not hope to reproduce for $r \geq 3$. Nevertheless, we have the following.

\begin{proposition}
\label{prop:ErdosStone}
Let $r \geq 2$ and suppose $\Pi_{\fin}^{(r)}$ is closed under $\oplus_r$, $\Pi_{\infty}^{(r)}$ is closed under $\otimes_2$, and the subgroup $\mathbb{G}_r$ of $(\mathbb{R}, +)$ generated by $\mathfrak{h}(\Pi_{\infty}^{(r)}\setminus\{1\})$ is not dense in $\mathbb{R}$. Then $\Pi_{\fin}^{(r)}=\Pi_{\infty}^{(r)} = \{1\} \cup \{1 - \frac{1}{k^{r-1}} : k \geq 1\}$.
\end{proposition}
\begin{proof}
By continuity and the fact that $\Pi_{\infty}^{(r)} \subseteq \overline{\Pi}_{\fin}^{(r)}$, $\Pi_{\infty}^{(r)}$ is also closed under $\oplus_r$. By Proposition \ref{prop:subring}, if $\Pi_{\infty}^{(r)}$ is closed under $\otimes_2$ then $\mathbb{G}_r$ is a subring of $\mathbb{R}$. It is well-known and easy to prove that a subgroup of $\mathbb{R}$ which is not dense must be cyclic. Hence $\mathbb{G}_r \cap (0, +\infty)$ has a smallest element $a$. As $a^2 \in \mathbb{G}_r$, it follows that $a = 1$. Thus $\mathbb{G}_r = \mathbb{Z}$, hence $\Pi_{\infty}^{(r)} = \{1\} \cup \{1 - \frac{1}{k^{r-1}} : k \geq 1\}$.

Now $0 \in \Pi_{\fin}^{(r)}$, and therefore the subsemigroup generated by $0$ under $\oplus_r$ belongs to $\Pi_{\fin}^{(r)}$. This subsemigroup is exactly $\{1 - \frac{1}{k^{r-1}} : k \geq 1\}$, and consequently $\Pi_{\fin}^{(r)}=\Pi_{\infty}^{(r)}$, proving the claim.
\end{proof}

As $\mathbb{G}_r$ is dense in $\mathbb{R}$ for all $r \geq 3$ by Proposition \ref{prop:dense}, Proposition \ref{prop:ErdosStone} only applies to $r=2$.

We know that $\Pi_{\fin}^{(2)}$ is closed under $\oplus_2$ by Theorem \ref{thm:local}, and $\Pi_{\infty}^{(2)}$ is closed under $\otimes_2$. However, Proposition \ref{prop:ErdosStone} does not give a new proof of \eqref{eq:revisit}, as I can not show that $\mathbb{G}_2$ is not dense without relying on the Erd\H os-Stone-Simonovits theorem. Furthermore my proof that $(\Pi_{\infty}^{(2)}, \otimes_2)$ is a semigroup relies on the ability to compute $\lambda(G)$, where $G$ is any $2$-graph. It would be interesting to find a different proof of these two claims, and possibly find other sets of hypotheses which imply \eqref{eq:revisit}.

\section*{\normalsize Acknowledgements}

I would like to thank my advisor Tibor Szab\'o for his support during the completion of this project and J\'anos K\"orner for sending me a copy of \cite{Korner88}.

\bibliographystyle{abbrv}
\bibliography{../../../MainBibFile/references}
\end{document}